\documentclass[11pt]{amsart}
\usepackage{hyperref}

\usepackage{enumerate,color,amssymb}

\usepackage[utf8]{inputenc}
\usepackage[T1]{fontenc}

\usepackage[mathscr]{eucal}

\usepackage[a4paper, hmargin={3cm,3cm}, centering]{geometry}

\makeatletter
\@namedef{subjclassname@2020}{
  \textup{2020} Mathematics Subject Classification}
\makeatother

\newtheorem{innercustomthm}{Theorem}
\newenvironment{customthm}[1]
  {\renewcommand\theinnercustomthm{#1}\innercustomthm}
  {\endinnercustomthm}
\newtheorem{lemma}{Lemma}[section]
\newtheorem{theorem}[lemma]{Theorem}
\newtheorem*{theorem*}{Theorem}
\newtheorem{corollary}[lemma]{Corollary}
\newtheorem{question}{Question}
\newtheorem{proposition}[lemma]{Proposition}
\newtheorem*{proposition*}{Proposition}
\newtheorem*{conjecture*}{Conjecture}

\newtheorem{conjecture}{Conjecture}

\newtheorem{problem}{Problem}

\newtheorem{remark}[lemma]{Remark}

\theoremstyle{remark}

\theoremstyle{definition}
\newtheorem*{definition*}{Definition}
\newtheorem*{remark*}{Remark}
\newtheorem*{remarks*}{Remarks}
\newtheorem*{claim*}{Claim}

\newtheorem*{example}{Example}

\newcommand{\nnorm}[1]{\lvert\!|\!| #1|\!|\!\rvert}

\title{}
\author{}
\date{} 

\begin{document}

\title[Total joint ergodicity for totally ergodic systems]{Total joint ergodicity for totally ergodic systems} 

\author{Andreas~Koutsogiannis and Wenbo~Sun}

\address[Andreas Koutsogiannis]{Department of Mathematics, Aristotle University of Thessaloniki, Thessaloniki, 54124, Greece} \email{akoutsogiannis@math.auth.gr}
\address[Wenbo Sun]{\textsc{Department of Mathematics, Virginia Tech, 225 Stanger Street, Blacksburg, VA, 24061, USA}}
\email{swenbo@vt.edu}

\begin{abstract}
Examining multiple ergodic averages whose iterates are integer parts of real valued polynomials for totally ergodic systems, we provide various characterizations of total joint ergodicity, meaning that an average converges to the ``expected'' limit along every arithmetic progression. In particular, we obtain a complete characterization when the number of iterates is at most two, and disprove a conjecture of the first author. We also improve a result of Frantzikinakis on joint ergodicity of Hardy field functions of at most polynomial growth for totally ergodic systems, which extends a conjecture of Bergelson-Moreira-Richter.
Our method is to first use the methodology of Frantzikinakis, which allows one to reduce the systems to rotations on abelian groups without using deep tools from ergodic theory, then develop formulas for integrals of exponential functions over subtori, and finally, compute exponential sums for integer parts of real polynomials.
\end{abstract}

\subjclass[2020]{Primary: 37A30; Secondary: 28D05, 11L03, 11L15}

\keywords{multiple ergodic averages, totally ergodic systems, equidistribution, Gowers-Host-Kra seminorms, Hardy field/ tempered/ Pfaffian/ real polynomial functions, Weyl sums}

\maketitle

\tableofcontents

\section{Introduction}\label{Sec:1} 

\subsection{The joint ergodicity and total joint ergodicity phenomena}

Let $(X,\mathcal{B},\mu)$ be a standard Borel probability space. If $T:X\to X$ is a measurable transformation which preserves the measure $\mu$ (i.e., $\mu(T^{-1}A)=\mu(A)$ for all $A\in\mathcal{B}$), then we say that the quadruple $(X,\mathcal{B},\mu,T)$ is a \emph{measure preserving system} (or just a \emph{system}). We will assume in what follows that the transformation $T$ is invertible.
We are interested in the $L^2$-limiting behavior, as $N\to\infty,$ of multiple ergodic averages of the form
\begin{equation}\label{E:main_expression}
\frac{1}{N}\sum_{n=1}^N T^{a_1(n)}f_1\cdot\ldots\cdot T^{a_\ell(n)}f_\ell,
\end{equation}
where the $(a_i(n))_n,$ $1\leq i\leq \ell,$ are 
suitable integer valued sequences, and the $f_i$'s are arbitrary bounded functions; for a positive integer $n,$ and a bounded function $f,$ 
$T^n$ denotes  the composition $T\circ \dots\circ T$ of $n$ copies of $T$, and $Tf(x):= f(Tx),$ $x\in X.$

In particular, we are interested in the case where the limit is the ``expected'' one.

\begin{definition*}
For a collection of sequences $a_1,\ldots,a_\ell:\mathbb{N}\to \mathbb{Z},$ and a system $(X,\mathcal{B},\mu,T),$ we say that $(a_1(n))_n,\ldots,$ $(a_\ell(n))_n$ are
\begin{itemize}
    \item \emph{jointly ergodic for $(X,\mathcal{B},\mu,T)$}, if for all functions $f_1,\ldots,f_\ell\in L^\infty(\mu)$ we have 
    \begin{equation*}\label{E:correct_limit}
        \lim_{N\to\infty}\frac{1}{N}\sum_{n=1}^N T^{a_1(n)}f_1\cdot\ldots\cdot T^{a_\ell(n)}f_\ell=\int_X f_1\;d\mu\cdot\ldots\cdot\int_X f_\ell\;d\mu,
    \end{equation*} 
    where the convergence is in $L^2(\mu).$ For $\ell=1,$ we say that $(a_1(n))_n$ is \emph{ergodic}.

    \item  \emph{totally jointly ergodic for $(X,\mathcal{B},\mu,T)$}, if for all functions $f_{1},\dots,f_{\ell}\in L^{\infty}(\mu)$, $W\in\mathbb{N},$ and $r\in\mathbb{Z},$ we have 
\[\lim_{N\to\infty}\frac{1}{N}\sum_{n=1}^{N}T^{a_{1}(Wn+r)}f_{1}\cdot\ldots\cdot T^{a_{\ell}(Wn+r)}f_{\ell}=\int_{X}f_{1}\,d\mu\cdot\ldots\cdot \int_{X}f_{\ell}\,d\mu,\]
where the convergence is in $L^2(\mu).$ For $\ell=1,$ we say that $(a_1(n))_n$ is \emph{totally ergodic}.\footnote{The notions of ``ergodicity'' and ``total ergodicity'' of a sequence are in general different. Let $(X,\mathcal{B},\mu,T)$ be a system. $T$ is \emph{ergodic} (in which case we also call the system $(X,\mathcal{B},\mu,T)$ \emph{ergodic}) if whenever $T^{-1}A=A,$ we have $\mu(A)\in \{0,1\}$ (i.e., there are no $T$-invariant sets of non-trivial measure), and $T$ is \emph{totally ergodic} (where we also call the system $(X,\mathcal{B},\mu,T)$ \emph{totally ergodic}) if $T^n$ is ergodic for all $n\in\mathbb{N}.$ Taking an ergodic transformation $T$ which is not totally ergodic, say $T^k$ is not ergodic for some $k\in\mathbb{N}$, we have from von Neumman's mean ergodic theorem that $(n)_n$ is an ergodic sequence, but $(kn+r)_n$ is not for any $r\in \mathbb{Z}$.}
\end{itemize}
\end{definition*}

Studying \eqref{E:main_expression} for $a_i(n)=in$, $1\leq i\leq \ell,$ Furstenberg managed to reprove (in \cite{Fu}) Szemer\'{e}di's theorem on arithmetic progressions.\footnote{ Every 
dense subset of natural numbers contains arbitrarily long arithmetic progressions.} In the same article, he also showed that the sequences $(n)_n,\ldots, (\ell n)_n$ are (totally) jointly ergodic for every weakly mixing system.\footnote{ We call a system $(X,\mathcal{B},\mu,T)$ \emph{weakly mixing} if $T$ is \emph{weakly mixing}, meaning that $T\times T$ is ergodic.} 

%Extending this last result, also introducing the PET (i.e., Polynomial Exhaustion Technique) induction scheme, which led to a number of far-reaching extensions of Szemer\'{e}di's theorem with applications in various areas of mathematics, 
A few years later, Bergelson showed (in \cite{Be}) that for all essentially distinct integer polynomials $p_1,\ldots,p_\ell$,\footnote{A polynomial $p\in \mathbb{Q}[x]$ is called \emph{integer} if $p(\mathbb{Z})\subseteq\mathbb{Z}$ (e.g., $p(n)=n(n+1)/2$). The non-constant polynomials $p_1,\ldots,p_\ell$ are \emph{essentially distinct} if for every $i\neq j,$ $p_i-p_j$ is non-constant.} the sequences $(p_1(n))_n,\ldots,(p_\ell(n))_n$ are (totally) jointly ergodic for every weakly mixing system. With this result not only he extended Furstenberg's convergence result on weakly mixing systems, but also introduced the PET (i.e., Polynomial Exhaustion Technique) induction scheme, which led to a number of far-reaching extensions of Szemer\'{e}di's theorem with applications in various areas of mathematics. 

It is natural to ask whether one can still get (total) joint ergodicity results by postulating less assumptions on the system while simultaneously strengthening the assumptions on the polynomial iterates. In this article, our main focus is to study total joint ergodicity properties for totally ergodic systems. To this end, it is convenient for us to have the following definition of independence of polynomials.

\begin{definition*}	
	Let $V\subseteq\mathbb{R}$ with $0\in V$ and $p_{1},\dots,p_{\ell}\in\mathbb{R}[x]$. We say that
 $p_{1},\dots,p_{\ell}$ are \emph{$V$-independent} if $c_{1}p_{1}+\dots+c_{\ell}p_{\ell}\in\mathbb{Q}[x]+\mathbb{R}, c_{i}\in V, 1\leq i\leq \ell,$ implies that $c_{1}=\dots=c_{\ell}=0$; otherwise, we call the $p_i$'s \emph{$V$-dependent}.
\end{definition*}

 Note that for any finite family of real polynomials $p_1,\ldots,p_\ell$, we have that
 \begin{equation}\label{pwr}
 \begin{split}
& \text{if $p_1,\ldots,p_\ell$ are $V$-independent}, 
\\& \text{then $p_1(W\cdot+r),\ldots,p_\ell(W\cdot+r)$ are $V$-independent for all $W\in\mathbb{N}$ and $r\in\mathbb{Z}$}.
 \end{split}
 \end{equation}

Frantzikinakis and Kra showed (in \cite{FK2}) 
 that if $p_1,\ldots,p_\ell$ are $\mathbb{R}\setminus \mathbb{Q}_\ast$-independent integer polynomials,\footnote{ In \cite{FK2} such polynomials are called \emph{rationally independent} (the independence condition there is presented in a different, equivalent, formulation).} then $(p_1(n))_n,\ldots,$ $(p_\ell(n))_n$ are (totally) jointly ergodic for every totally ergodic system.
  Moreover, in that result, the independence assumption on the polynomial family is also necessary (see comments after \cite[Theorem~1.1]{FK2}, and Theorem~\ref{cor3} below).\footnote{  For characterizations of (total) joint ergodicity for integer polynomials, see also \cite{DFKS, DKS, FraKu, FraKu2}. Actually, all the results we stated up to this point were dealing with the notion of ``joint ergodicity''; it is because of \eqref{pwr} that we can upgrade their conclusions to ``total joint ergodicity''.}

Karageorgos and the first author showed (in \cite{KK}) that if $p_1,\ldots,p_\ell$ are $\mathbb{R}$-independent real polynomials,\footnote{ In \cite{KK} such polynomials are called \emph{strongly independent}.} then $([p_1(n)])_n,\ldots,([p_\ell(n)])_n$ are (totally) jointly ergodic for every ergodic system.\footnote{See \cite{K} for a generalization of this result for sequences of variable real polynomials.} 
%What lies in the heart of the proof of this result, is an application of Weyl's criterion (see Remark~\ref{rem1} (ii)) to obtain an equidistribution result on nilmanifolds for ergodic nilrotations (\cite[Proposition~4.3]{KK}). 

A conjecture\footnote{ This conjecture was never recorded  anywhere formally.} of the first author of this article, following the philosophy of the results that were stated above, i.e., postulating a stronger assumption on the system while simultaneously assuming a weaker independence condition on the polynomial iterates, is that one should expect the following.
 
 \begin{conjecture}\label{Conjecture:1}
 	For $\ell\in \mathbb{N},$ let $p_1,\ldots,p_\ell\in \mathbb{R}[x].$ $p_1,\ldots,p_\ell$  are $\mathbb{R}\backslash\mathbb{Q}_\ast$-independent if, and only if, $([p_1(n)])_n,$ $\ldots,([p_\ell(n)])_n$ are totally jointly ergodic for every totally ergodic system.
 \end{conjecture}
 
Our first result is a positive answer to the sufficiency part of this conjecture.

\begin{customthm}{A}\label{cor1}
For $\ell\in \mathbb{N},$ let $p_1,\ldots,p_\ell\in \mathbb{R}[x].$ If $p_{1},\dots,p_{\ell}$ are $\mathbb{R}\backslash\mathbb{Q}_{\ast}$-independent, then $([p_{1}(n)])_n,\dots,$  $([p_{\ell}(n)])_n$ are totally jointly ergodic for every totally ergodic system.
\end{customthm}

In fact, we prove a result more general than
Theorem~\ref{cor1}, namely, Theorem~\ref{T:TP}, which deals with combinations of polynomials and tempered Hardy field functions.
Our proof relies on an improvement of a recent result of Frantzikinakis (\cite[Theorem~1.9 (i)]{F2}--see Theorem~\ref{T:1.7}) on joint ergodicity, which also implies a generalization of \cite[Conjecture~6.1]{BMR}. We defer the details to Subsection~\ref{s112}.

As we mentioned above, the necessity part of Conjecture \ref{Conjecture:1} holds when $p_{1},\dots,p_{\ell}$ are integer polynomials. 
In the case $\ell=1$, we show that Conjecture~\ref{Conjecture:1} holds in general (see Corollary~\ref{cor55} which is derived from Theorem~\ref{cor3}). On the other hand, we (somewhat surprisingly) discover that the necessity part of Conjecture \ref{Conjecture:1} is false in general. In fact, we obtain necessary and sufficient conditions for total joint ergodicity of two real polynomials $p_{1},p_{2}$ with $p_{1}(0)=p_{2}(0)=0,$ which is the main result of the article.

\begin{customthm}{B}\label{cor6}
	For all $p_{1},p_{2}\in\mathbb{R}[x]$ with $p_{1}(0)=p_{2}(0)=0$, $([p_{1}(n)])_n, ([p_{2}(n)])_{n}$ are totally jointly ergodic for all totally ergodic systems if, and only if, the following holds:
	\begin{itemize}
	    \item[(i)] $p_{1},p_{2}$ are $\mathbb{R}\backslash\mathbb{Q}_{\ast}$-independent; or
	    \item[(ii)] $p_1=f+cg,\; p_2=u(f+(c+1)g)$ for some $f,g\in\mathbb{Q}[x],$ where $f$ is not a multiple of $g$,\footnote{ By this we mean that there is no $s\in\mathbb{Q}$ such that $f=sg.$}  $f(0)=g(0)=0,$ $g\not\equiv 0,$ $c\in \mathbb{R}\backslash\mathbb{Q},$ with $u=\pm 1,$ such that $g$ is an integer polynomial.
	\end{itemize}
\end{customthm}

We remark that $p_{1}(0)=p_{2}(0)=0$ is an arguably common (and sometimes necessary) assumption in the study of multiple recurrence properties. However, it is  natural to ask whether one can get a characterization for polynomials with non-trivial constant terms (see Question~\ref{Q:2} below). Unlike the case of integer polynomials, the constant terms bring many additional difficulties in the study of the equidistribution properties for iterates of integer parts of real polynomials which are not integer ones. Indeed, it can happen that for two polynomials $p_1, p_2,$ in some totally ergodic system and for some $a\in\mathbb{R},$ $([p_1(n)])_n,$ $([p_2(n)])_n$ behave differently than $([p_1(n)])_n,$ $([p_2(n)+a])_n$ (see Proposition~\ref{apd}).

%We leave the proof of Proposition \ref{apd} for the appendix (Section~\ref{app}).

\medskip

Theorem \ref{cor6} indicates that a characteristic condition for more than two polynomials could be very intricate (see Problem~\ref{P:1}).
 We also note that our method doesn't generalize to more than two terms (see Proposition~\ref{case2}, the argument of which works only for $\ell=2$). 
Nevertheless, we were able to obtain a partial characterization in the general case.

\begin{customthm}{C}\label{cor2}
	For $\ell\in\mathbb{N},$ let  $p_{1},\dots,p_{\ell}\in\mathbb{R}[x]$. 
	If there exists a nonempty subset $\{i_{1},\dots,i_{k}\}$ $\subseteq\{1,\dots,\ell\}$ such that $p_{i_{1}},\dots,p_{i_{k}}$ are $\mathbb{R}\backslash\mathbb{Q}_{\ast}$-dependent and all the irrational polynomials\footnote{ By this we mean polynomials which don't belong to $\mathbb{Q}[x]+\mathbb{R}.$} in $p_{i_{1}},\dots,p_{i_{k}},$ if any, are $\mathbb{Q}$-independent, 
	then there exists a totally ergodic system $(X,\mathcal{B},\mu,T)$  such that $([p_{1}(n)])_n,\ldots,$ $([p_{\ell}(n)])_n$ are not 
	totally jointly ergodic for $(X,\mathcal{B},\mu,T)$.
\end{customthm}

 By combining Theorems~\ref{cor1} and ~\ref{cor2}, we deduce that for  two special classes of polynomials, the condition in Conjecture~\ref{Conjecture:1} is characteristic.
	
\begin{customthm}{D}\label{cor3}
For $\ell\in\mathbb{N},$ let  $p_{1},\dots,p_{\ell}\in\mathbb{Q}[x]+\mathbb{R}$ (resp. $p_{1},\dots,p_{\ell}\in\mathbb{R}[x]$ so that all the irrational polynomials in $p_1,\dots,p_\ell,$ if any, are $\mathbb{Q}$-independent). $p_{1},\dots,p_{\ell}$ are $\mathbb{R}\backslash\mathbb{Q}_{\ast}$-independent if, and only if, $([p_{1}(n)])_n,\ldots,([p_{\ell}(n)])_n$ are totally jointly ergodic for every totally ergodic system.
\end{customthm}
	
Noticing that the independence condition for integer polynomials of Frantzikinakis and Kra is equivalent to  $\mathbb{R}\backslash\mathbb{Q}_{\ast}$-independence,  we have that Theorem~\ref{cor3} generalizes the main theorem of \cite{FK2} as well. 

\subsection{Improvement of a result on joint ergodicity}\label{s112}
One of the main ingredients of this article is a recent criterion of Frantzikinakis (\cite[Theorem~1.1]{F2}--see Theorem~\ref{T:1.1_Nikos}). In \cite{F2}, Frantzikinakis provided an approach to study joint ergodicity properties without using  either the Host-Kra theory of characteristic factors  or equidistribution results on nilmanifolds, the combination of which is arguably the most influential approach to characterize the convergence of \eqref{E:main_expression} to the expected limit when the sequences $a_{1},\dots,a_{\ell}$ are  coming from suitable classes of functions.
Using this result, not only he obtained new convergence results, but also provided simpler proofs for most of the known ones.

 As in \cite{F2}, we assume that all Hardy fields $\mathcal{H}$ considered (see Subsection~\ref{SS:HT} for the definition  of a Hardy field and of functions in such a field) have the property:
\begin{equation}\label{E:Hardy}
\text{If}\; a,b\in \mathcal{H}, \;\text{then}\; a\circ b^{-1}\in \mathcal{H}, \;\text{and}\;a(\cdot +h)\in \mathcal{H}\;\text{for all}\;h\geq 0.
\end{equation}

For two functions $a,b:(x_0,\infty)\to\mathbb{R},$ we write $a\prec b$ if $|a(x)|/|b(x)|\to 0$ as $x\to\infty.$

\begin{definition*}
We say that a Hardy field function $a:(x_0,\infty)\to\mathbb{R},$ $x_0\geq 0,$
\begin{itemize}
    \item [(i)] is \emph{tempered} (and we write $a\in \mathcal{T}$) if there exists $k\in\mathbb{N}:$ $x^{k-1}\log x\prec a(x)\prec x^k.$
    
    \item[(ii)] \emph{stays logarithmically away from rational polynomials} if $\log x\prec a(x)-p(x)$ for all $p\in \mathbb{Q}[x].$\footnote{  Or, equivalently, for all $p\in \mathbb{Q}[x]+\mathbb{R}.$ Notice here the similarities between this condition and Weyl's criterion on equidistribution (in $\mathbb{T}$) when $a$ is restricted to $\mathcal{P}$ (see Remark~\ref{rem1} (ii)).} 
\end{itemize}
\end{definition*}

When the functions $a_1,\ldots,a_\ell$ and their differences are in $\mathcal{T}+\mathcal{P},$\footnote{ As $\mathcal{T}$ denotes the class of tempered functions, and $\mathcal{P}$ the class of real polynomials, $\mathcal{T} + \mathcal{P}$ denotes the class of all linear
combinations of tempered functions and real polynomials. The reason why it is natural to work with these classes, is that differences of them (i.e., derivatives) fall into the same classes, allowing one to study the corresponding  averages by using variations of the PET induction, via van der Corput's lemma (approach that was initiated in \cite{Be}). For (total) joint ergodicity results for iterates coming from a Hardy field, see \cite{BMR, DKS3, F3, F4, Ts}.} Frantzikinakis also showed the following result which resolves \cite[Conjecture~6.1]{BMR}.

\begin{theorem}[Theorem~1.9 (i), \cite{F2}]\label{T:1.7}
For $\ell\in\mathbb{N},$ let $a_1,\ldots,a_\ell:(x_0,+\infty)\to\mathbb{R}$ be functions from a Hardy field $\mathcal{H}$ that satisfies \eqref{E:Hardy}. Suppose that
the $a_i$'s and their differences are in $\mathcal{T}+\mathcal{P}$ and every non-trivial linear combination
of them with at least one irrational coefficient, stays logarithmically away from rational polynomials. %\footnote{ The notion ``stays logarithmically away from rational polynomials'' will also be defined in Subsection~\ref{SS:HT}.}
Then, the sequences $([a_1(n)])_n,\ldots,([a_\ell(n)])_n$ are jointly ergodic for every totally ergodic system.
\end{theorem}

Following this newest approach of Frantzikinakis, and applying an extra twist in \cite[Lemma~6.2]{F2} (see Proposition~\ref{thm1}), we obtain the following improvement of the previous result.

\begin{customthm}{E}\label{T:TP}
For $\ell\in\mathbb{N},$ let $a_1,\ldots,a_\ell:(x_0,+\infty)\to\mathbb{R}$ be functions from a Hardy field $\mathcal{H}$ that satisfies \eqref{E:Hardy}. Suppose that
the $a_i$'s and their differences are in $\mathcal{T}+\mathcal{P}$ and every non-trivial linear combination
of them with irrational or zero coefficients, stays logarithmically away from rational polynomials.
 Then, the sequences $([a_1(n)])_n,\ldots,([a_\ell(n)])_n$ are jointly ergodic for every totally ergodic system.
\end{customthm}

%Notice that both Theorems~\ref{T:1.7} and ~\ref{T:TP}, because of \eqref{pwr}, imply total joint ergodicity. 
We prove Theorem~\ref{T:TP} in Section \ref{Sec:SCP}. In particular,
Theorem~\ref{T:TP} implies Theorem \ref{cor1}.

\begin{example}
To see why Theorem~\ref{T:TP} improves Theorem~\ref{T:1.7}, consider $p_1(n)=n^3+an^2+a^2n,$ and $p_2(n)=n^2+an,$ for $a\in\mathbb{R}\backslash\mathbb{Q}.$ Then, $p_1(n)-ap_2(n)=n^3\in \mathbb{Q}[n],$  so Theorem~\ref{T:1.7} does not provide any info on the joint ergodicity of the sequences $([p_1(n)])_n,$ $([p_2(n)])_n$. On the other hand, if $b_1p_1(n)+b_2p_2(n)\in \mathbb{Q}[n],$ for $b_1,b_2\in \mathbb{R}\backslash\mathbb{Q}_\ast,$ we get $b_1=b_2=0,$  hence Theorem~\ref{T:TP} implies that $([p_1(n)])_n,$ $([p_2(n)])_n$ are indeed jointly ergodic for every totally ergodic system.
\end{example}

\begin{remark}
Notice that if a function $a(x)$ stays logarithmically away from rational polynomials, then the same is true for the function $a(Wx+r)$ for all $W\in\mathbb{N},$ $r\in\mathbb{Z}.$  So, in Theorems~\ref{T:1.7} and ~\ref{T:TP} (thus in the previous example as well) we actually have total joint ergodicity for the sequences of interest.
\end{remark}

\subsection{Questions and Problems}
For Hardy field iterates coming from $\mathcal{T}+\mathcal{P},$ one can immediately ask the following question.   

\begin{question}
What is a characteristic property in Theorem~\ref{T:TP}  so that the sequences $([a_1(n)])_n,$ $\ldots, ([a_\ell(n)])_n$ are (totally) jointly ergodic for every totally ergodic system?
\end{question}

As we mentioned before, we expect the answer to this question to be very hard as it is not clear how to answer it even for $\ell=3,$ when we deal with total joint ergodicity and we restrict our study to $\mathcal{P}$ (see Problem~\ref{P:1} below).

 Restricting to $\mathcal{P},$ while our method uses  the assumption $p_1(0)=p_2(0)=0$ crucially, it is reasonable to ask the following.

\begin{question}\label{Q:2}
Can the assumption $p_1(0)=p_2(0)=0$ be dropped in Theorem~\ref{cor6}?
\end{question}

Since our approach to Theorem~\ref{cor6} cannot be extended to $\ell\geq 3$, one can state the following.

\begin{problem}\label{P:1}
Extend Theorem~\ref{cor6} for $\ell\geq 3.$
\end{problem}

While 
in general the notions of  ``ergodicity'' and ``total ergodicity'' are different, we didn't manage to find a pair of polynomial sequences that are jointly ergodic but not totally jointly ergodic. %(The following question can be stated for any $\ell\geq 3$.)

\begin{question}\label{P:2}
Does Theorem~\ref{cor6} characterize the notion of ``joint ergodicity for all totally ergodic systems'' for $([p_{1}(n)])_n, ([p_{2}(n)])_{n},$ where $p_{1},p_{2}\in\mathbb{R}[x]$ with $p_{1}(0)=p_{2}(0)=0$?
\end{question}

In case the answer to the previous question is negative, a natural, follow-up problem is the following (which can be stated for any $\ell\geq 3$ sequences as well).

\begin{problem}
Similarly to Theorem~\ref{cor6}, find a characterization to the ``joint ergodicity for all totally ergodic systems'' notion.
\end{problem}

\subsection{Strategy and organization of the paper}   
We provide some background material in Section~\ref{sec:2}.
In Section~\ref{Sec:SCP}, we prove all the results stated in Section~\ref{Sec:1} except Theorem~\ref{cor6}.
In fact, in Sections~\ref{sec:2} and ~\ref{Sec:SCP} we obtain stronger versions of these results for a fixed system (instead of all systems) and also deal with $W!$-joint ergodic property (defined in Section \ref{Sec:SCP} as well), and then derive the aforementioned results as corollaries. By applying an additional twist to the argument in \cite[Lemma~6.2]{F2}, we first prove Theorem~\ref{T:TP},  an enhancement of Frantizikinakis' result. This enables us to reduce Theorems~\ref{cor1}, \ref{cor2} and \ref{cor3}  to the case where the system is a rotation on an abelian group; we then use Weyl's equidistribution theorem to prove them.

%In Subsections~\ref{SSJE} and ~\ref{Ss:TB} we also study joint ergodicity properties for a fixed system (in the first one for expressions from $\mathcal{T}+\mathcal{P}$ and $\ell\in \mathbb{N}$ that we will use in the proof of Theorem~\ref{T:TP}, and polynomial ones for $\ell=2$  that will be used in the proof of Theorem~\ref{cor6} respectively). 

Sections \ref{sec:4}, \ref{Section:6} and \ref{sec:7} are devoted to the proof of Theorem \ref{cor6}. In Section \ref{sec:4}, we prove some formulas for the integral of exponential functions over subtori of $\mathbb{T}^{2}$ and $\mathbb{T}^{3}$. In Section~\ref{Section:6}, we first derive some equidistribution properties for polynomials, and then use the formulas obtained in  Section \ref{sec:4}  to study limits of exponential sums, as
\begin{equation}\nonumber
    \lim_{N\to\infty}\frac{1}{N}\sum_{n=1}^{N}e\Bigl(\sum_{i=1}^{2}t_{i}[p_{i}(n)]\Bigr).
\end{equation}
Using the tools developed in the previous sections, we prove Theorem~\ref{cor6} by splitting it into various cases in Section~\ref{sec:7}.

Finally, in Section \ref{app} (the appendix), we provide  two explicit examples to demonstrate the complexity of the behavior of joint ergodicity for $([p_{1}(n)])_n$, $([p_{2}(n)])_n$ with $p_{1},p_{2}\in\mathbb{R}[x]$. %In particular, we prove Proposition~\ref{apd} in that section.

\subsection*{Notation}
 With $\mathbb{N}=\{1,2,\ldots\},$ $\mathbb{Z},$ $\mathbb{Q},$ $\mathbb{R},$ and $\mathbb{C}$ we denote the sets of natural, integer, rational, real, and complex numbers respectively.  For $m\in \mathbb{N},$ $\mathbb{T}^m = (\mathbb{R}/\mathbb{Z})^m$ denotes the $m$ dimensional torus 
and $(a(n))_n$ denotes a sequence indexed over the natural numbers. For every $A\subseteq \mathbb{R},$ we write $A_\ast:=A\backslash\{0\}.$  
$[\cdot]$ (resp. $\{\cdot\}$) is the integer part (resp. fractional part) function, where, for all $x\in \mathbb{R},$ we have $x=[x]+\{x\}.$ Finally, $e(t):=e^{2\pi i t},$ $t\in \mathbb{R}.$

\section{An improvement of Frantzikinakis' result}\label{sec:2}

\subsection{Hardy field functions}\label{SS:HT}
%{\color{blue} We start by defining the class of Hardy field functions.}

Let $R$ be the collection of equivalence classes of real valued functions defined on some halfline $(x_0,\infty),$ $x_0\geq 0,$ where two functions that agree eventually are identified. These classes are called \emph{germs} of functions.  A \emph{Hardy field} is a subfield of the ring $(R, +, \cdot)$ that is closed under differentiation. Here, we use the word \emph{function} when we refer to elements of $R$ (understanding that all the operations defined and statements made for elements of $R$ are considered only for sufficiently large $x\in \mathbb{R}$).

 As it was mentioned in the introduction, following~\cite{F2}, we assume that all Hardy fields $\mathcal{H}$ considered satisfy the property:
\begin{equation*}
\text{If}\; a,b\in \mathcal{H}, \;\text{then}\; a\circ b^{-1}\in \mathcal{H}, \;\text{and}\;a(\cdot +h)\in \mathcal{H}\;\text{for all}\;h\geq 0.\footnote{ Such a Hardy field is the one of Pfaffian functions which contains all the logarithmico-exponential functions ($a$ is a \emph{logarithmico-exponential Hardy field function} if it belongs to a Hardy field of real valued functions and it is defined on some $(x_0,+\infty),$ $x_0\geq 0,$ by a finite combination of symbols $+, -, \times, \div, \sqrt[n]{\cdot}, \exp, \log$ acting on the real variable $x$ and on real constants). For more on Hardy field functions, see \cite{Bos, F3, F4, Hard}.}
\end{equation*}
  In particular, we will deal with Hardy field functions form $\mathcal{T} + \mathcal{P},$ where we recall that $\mathcal{T}$ denotes the class of tempered functions (i.e., functions $a$ that, for some $k\in\mathbb{N},$ satisfy $x^{k-1}\log x\prec a(x)\prec x^k$), and $\mathcal{P}$ is the class of real polynomials.

\subsection{Equidistribution on (sub)tori}

Let $Y$ be a finite-dimensional subtorus (i.e., a subgroup of a finite-dimensional torus). We say that a sequence $(x_n)_n\subseteq Y$ is \emph{equidistributed} on $Y$ if for every complex-valued Riemann integrable function $f$ on $Y$ we have 
\begin{equation*}\label{E:W}
\lim_{N\to\infty} \frac{1}{N} \sum_{n=1}^{N} f(x_n)=\int_Y f(x)\;dm_Y(x),\end{equation*} where $m_Y$ is the Haar measure on $Y.$

The following is a consequence of a special case of  \cite[Theorem 1.9]{GT} (see also \cite{Leib}).
   
\begin{proposition}\label{green}
Let $m,r\in\mathbb{N},$ $v_{i}=(v_{i,1},\dots,v_{i,m})\in\mathbb{Q}^{m}$ for $1\leq i\leq r,$ $H$ be the set of $(u_{1},\dots,u_{m})\in\mathbb{R}^{m}$ such that $u_{1}v_{i,1}+\dots+u_{m}v_{i,m}=0$ for all $1\leq i\leq r,$ $Y$ be the set of $(\{x_{1}\},\dots,\{x_{m}\})\in\mathbb{T}^{m}$ for some $(x_{1},\dots,x_{m})\in H,$ and $p=(p_{1},\dots,p_{m})\colon\mathbb{Z}\to\mathbb{R}^{m}$ be a vector valued polynomial taking values in $H$. If $(\{p(n)\})_{n}$ is not equidistributed on $Y,$ then there exist $k=(k_{1},\dots,k_{m})\in\mathbb{Z}^{m}$ which does not belong to $\emph{\text{span}}_{\mathbb{Q}}\{v_{1},\dots,v_{r}\}$\footnote{$\text{span}_{A}\{v_{1},\dots,v_{r}\}$ denotes all the linear combinations of $v_1,\ldots,v_r$ with coefficients from $A,$ $A\subseteq \mathbb{R}$.} and some $t\in\mathbb{R}$ such that $k_{1}p_{1}(n)+\dots+k_{m}p_{m}(n)\equiv t \mod 1$  for all $n\in\mathbb{Z}$. 
\end{proposition}

\begin{remark}\label{rem1}
\emph{(i)} If we impose the additional assumption $p(0)=(0,\dots,0)$ in Proposition~\ref{green}, then we have that $k_{1}p_{1}(n)+\dots+k_{m}p_{m}(n)\in\mathbb{Z}$ for all $n\in\mathbb{Z}$. 

\emph{(ii)} Proposition \ref{green} also extends to the case where $r=0$, $H=\mathbb{R}^{m}$ and $Y=\mathbb{T}^{m}$ ($\emph{\text{span}}_{\mathbb{Q}}\{v_{1},\dots,v_{r}\}$ is understood as  the singleton $\{(0,\dots,0)\}$). This is in fact Weyl's criterion (see \cite{W}, or \cite{Nid}):  For any sequence $(x_n)_n\subseteq \mathbb{R}^m$, $(\{x_n\})_n$\footnote{ If $x_n=(x_1(n),\ldots,x_m(n)),$ by $(\{x_n\})_n$ we mean $(\{x_1(n)\},\ldots,\{x_m(n)\})_n.$} is equidistributed on $\mathbb{T}^{m}$ if, and only if, for all $k\in \mathbb{Z}^m\setminus\{(0,\ldots,0)\},$ we have that $\frac{1}{N}\sum_{n=1}^N e^{2\pi i \langle k,x_n\rangle}\to 0,$ as $N\to\infty,$ where $\langle\cdot,\cdot\rangle$ denotes the standard inner product.
\end{remark}

\subsection{Frantizikinakis' criterion on joint ergodicity}\label{ss:Frantz}

As we mentioned in the introduction, the proof of Theorem~\ref{T:TP} (as the one of Theorem~\ref{T:1.7}) is primarily based on Frantzikinakis' main result from \cite{F2} (\cite[Theorem~1.1]{F2}--see Theorem~\ref{T:1.1_Nikos} below). In order to state it we need some additional definitions. 

Following \cite{HK99}, given a system $(X,\mathcal{B},\mu,T),$ we inductively define the \emph{Gowers-Host-Kra seminorms} $\nnorm{\cdot}_k$:\footnote{ $\nnorm{\cdot}_0,$ which is not a seminorm, is defined for convenience.} For $f\in L^\infty(\mu),$ we let
\[\nnorm{f}_0:=\int f\;d\mu,\;\text{and for}\;k\in\mathbb{N},\; \nnorm{f}_k^{2^k}:=\lim_{N\to\infty}\frac{1}{N}\sum_{n=1}^N \nnorm{\bar{f}\cdot T^n f}_{k-1}^{2^{k-1}}.\]

Next, we define the \emph{spectrum} of a transformation.

\begin{definition*}
For a system $(X,\mathcal{B},\mu,T)$ we let 
\[\text{Spec}(T):=\big\{t\in [0,1):\;Tf=e(t)f\;\text{for some nonzero}\;f\in L^2(\mu)\big\}.\]
\end{definition*}

\begin{remark}\label{R:te}
Notice that, for a totally ergodic transformation $T,$ $\text{Spec}(T)\subseteq [0,1)\setminus \mathbb{Q}_\ast,$ as $T$ cannot have an eigenvalue $\lambda\neq 1$ which is a root of unity.
\end{remark}

The following notions of ``good for seminorm estimates'' and ``good for equidistribution'' characterize joint ergodicity for ergodic systems.

\begin{definition*}
For a collection of sequences $a_1,\ldots,a_\ell:\mathbb{N}\to\mathbb{Z},$ and a system $(X,\mathcal{B},\mu,T),$ we say that $(a_1(n),\ldots,a_\ell(n))_n$ is
\begin{enumerate}
    \item[(i)] \emph{good for seminorm estimates} for  $(X,\mathcal{B},\mu,T)$, if there exists $s\in \mathbb{N}$ such that if $f_1,\ldots,f_\ell\in L^\infty(\mu)$ and $\nnorm{f_{i_0}}_s=0$ for some $1\leq i_0\leq \ell,$ then 
    \begin{equation*}
        \lim_{N\to\infty}\frac{1}{N}\sum_{n=1}^N T^{a_1(n)}f_1\cdot\ldots\cdot T^{a_{i_0}(n)}f_{i_0}=0,
    \end{equation*}
    where the convergence is in $L^2(\mu).$ It is \emph{good for seminorm estimates}, if it is good for seminorm estimates for every ergodic system.
    \item[(ii)] \emph{good for equidistribution} for $(X,\mathcal{B},\mu,T)$, if for all $t_1\ldots,t_\ell\in \text{Spec}(T),$ not all of them $0,$ we have 
    \begin{equation}\label{E:equid}
        \lim_{N\to\infty}\frac{1}{N}\sum_{n=1}^N e\big(t_1 a_1(n)+\ldots+t_\ell a_\ell(n)\big)=0.
    \end{equation}
    \emph{It is good for equidistribution}, if it is good for equidistribution for every system, or, equivalently, if \eqref{E:equid} holds for all $t_1,\ldots,t_\ell\in [0,1),$ not all of them $0.$
\end{enumerate}
\end{definition*}

\begin{theorem}[Theorem~1.1, \cite{F2}]\label{T:1.1_Nikos}
For $\ell\in \mathbb{N},$ let  $a_1,\ldots,a_\ell:\mathbb{N}\to\mathbb{Z}$ be sequences. $(a_1(n))_n,\ldots,(a_\ell(n))_n$ are jointly ergodic for an ergodic system $(X,\mathcal{B},\mu,T)$ if, and only if, $(a_1(n),\ldots,a_\ell(n))_n$ is good for seminorm estimates and equidistribution for $(X,\mathcal{B},\mu,T).$
\end{theorem}

For the class of Hardy field functions from $\mathcal{T}+\mathcal{P},$ %(that we defined in Subsection~\ref{s112}),
we have the following result which follows from \cite[Proposition~6.5]{F2}.

\begin{proposition}[Proposition~6.5, \cite{F2}]\label{6.5_Nikos}
For $\ell\in \mathbb{N},$ let $a_1,\ldots, a_\ell: (x_0, \infty) \to \mathbb{R}$ be functions from a Hardy field that satisfies \eqref{E:Hardy} such that
the $a_i$'s and their pairwise differences are non-constant functions in $\mathcal{T}+\mathcal{P}$. Then $([a_1(n)], \ldots , [a_\ell(n)])_n$ is good for seminorm estimates.
\end{proposition}

\begin{remark}\label{R_proof}
To show our joint ergodicity results, by Theorem~\ref{T:1.1_Nikos}, via Proposition~\ref{6.5_Nikos}, it only suffices to verify the ``good for equidistribution'' property, which is a Weyl-type sum, hence, we only have to check equidistribution on (sub)tori.
\end{remark}

In particular, for Hardy field functions of at most polynomial growth,\footnote{ A function $a:(x_0,\infty)\to\mathbb{R}$ has \emph{at most polynomial growth} if there exists $k\in \mathbb{N}$ such that $a(x)\prec x^k.$}  hence for functions from $\mathcal{T}$ as well, we have the following result of Boshernitzan.

\begin{theorem}[\cite{Bos}]\label{T:Bos}
Let $a: (x_0, \infty) \to\mathbb{R}$ be a Hardy field function with
at most polynomial growth. $(a(n))_n$ is equidistributed on $\mathbb{T}$ if, and only
if, it stays logarithmically away from rational polynomials.
\end{theorem}

\subsection{The improved condition on joint ergodicity}\label{SSJE}
Using Theorem \ref{T:Bos}, Theorem~\ref{T:TP} will follow from Proposition~\ref{thm1_TP} (see below). The latter provides a sufficient condition of (total) joint ergodicity for a specific system (its proof follows the philosophy of \cite[Lemma~6.2]{F2}). While our assumption is weaker than that of \cite[Lemma~6.2]{F2} (we restrict the coefficients of the linear combinations to the spectrum of the transformation $T$ and not the whole set of irrational numbers), our conclusion is stronger.

It is more convenient for us, instead of dealing with $\text{Spec}(T)$, to work with the set
\[S(T):=\text{Spec}(T)+\mathbb{Z}.\]
Notice that, because of Remark~\ref{R:te}, $S(T)\subseteq (\mathbb{R}\setminus \mathbb{Q})\cup \mathbb{Z}.$

\begin{proposition}\label{thm1_TP}
Let $(X,\mathcal{B},\mu,T)$ be a totally ergodic system and $a_1,\ldots,a_\ell:(x_0,+\infty)\to\mathbb{R}$ be functions from a Hardy field that satisfies \eqref{E:Hardy}. Suppose that
the $a_i$'s and their differences are in $\mathcal{T}+\mathcal{P}$ and that every non-trivial linear combination
of them, with coefficients from $S(T)\backslash\mathbb{Z}_\ast$, stays logarithmically away from rational polynomials. Then, the sequences $([a_1(n)])_n,\ldots, ([a_\ell(n)])_n$ are totally jointly ergodic for $(X,\mathcal{B},\mu,T).$
\end{proposition}

To show Proposition~\ref{thm1_TP}, we follow the arguments of \cite[Lemma 6.2]{F2}. In our proof, we delete the $t_{i}$'s that are equal to $0$ at the beginning, and then continue working under the assumption that all the $t_{i}$'s are nonzero (in which case the claim holds and leads to the required strengthening of \cite[Proposition~6.4]{F2} and Theorem~\ref{T:1.7}).

\begin{proof}[Proof of Proposition~\ref{thm1_TP}]
We first claim that 
for any $s\leq \ell$, $t_{1},\dots,t_{s}\in S(T)\backslash\mathbb{Z},$ and for any Riemann integrable function $G\colon \mathbb{T}^{s}\to \mathbb{C}$, we have that
\begin{equation}\label{3515}
    \lim_{N\to\infty}\frac{1}{N}\sum_{n=1}^{N}e\Big(\sum_{i=1}^{s}t_i a_{i}(n)\Big)G(a_{1}(n),\dots,a_{s}(n))=0.
\end{equation}
To see this, approximating $G$ from below and above by continuous functions, and then by trigonometric polynomials, we may assume without loss of generality that $G(x_{1},\dots,x_{s})=e(k_{1}x_{1}+\dots+k_{s}x_{s}),$ $x_{1},\dots,x_{s}\in\mathbb{T}$ for some $k_{1},\dots,k_{s}\in\mathbb{Z}$.

The left-hand side of (\ref{3515}) equals to
\begin{equation}\label{3516}
    \lim_{N\to\infty}\frac{1}{N}\sum_{n=1}^{N}e\Big(\sum_{i=1}^{s}(t_{i}+k_{i})a_{i}(n)\Big).
\end{equation}
Clearly, each $t_{i}+k_{i}$ belongs to $S(T)\backslash\mathbb{Z}$. Since every non-trivial linear combination
of $a_1,\ldots,a_\ell,$ with coefficients from $S(T)\backslash\mathbb{Z}_\ast$, stays logarithmically away from rational polynomials, we have from Theorem~\ref{T:Bos} that (\ref{3516}) equals to $0$ by Weyl's equidistribution criterion.

To prove the statement, by Remark~\ref{R_proof}, 
it suffices to show that for all $t_{1},\dots,t_{\ell}\in S(T)\backslash\mathbb{Z}_{\ast}$ not all equal to 0, we have that 
\begin{equation*}\label{3517}
    \lim_{N\to\infty}\frac{1}{N}\sum_{n=1}^{N}e\Big(\sum_{i=1}^{\ell}t_{i}[a_{i}(n)]\Big)=0.
\end{equation*}
We may assume without loss of generality that $t_{1},\dots,t_{m}\neq 0$ and $t_{m+1}=\dots=t_{\ell}=0$ for some $1\leq m\leq \ell$.  Then, applying the claim for $s=m$ and $G(x_{1},\dots,x_{m}):=e(-\{x_{1}\}t_{1}-\ldots-\{x_{m}\}t_{m}),$ $x_{1},\dots,x_{m}\in\mathbb{T}$, we have that 
\begin{eqnarray*}
    \lim_{N\to\infty}\frac{1}{N}\sum_{n=1}^{N}e\Big(\sum_{i=1}^{\ell}t_{i}[a_{i}(n)]\Big) & = & \lim_{N\to\infty}\frac{1}{N}\sum_{n=1}^{N}e\Big(\sum_{i=1}^{m}t_{i}[a_{i}(n)]\Big) \\
    & = & \lim_{N\to\infty}\frac{1}{N}\sum_{n=1}^{N} e\Big(\sum_{i=1}^{m}t_{i}a_{i}(n)\Big)e\Big(-\sum_{i=1}^{m}t_{i}\{a_{i}(n)\}\Big) \\
    & = & \lim_{N\to\infty}\frac{1}{N}\sum_{n=1}^{N}e\Big(\sum_{i=1}^m t_i a_i(n)\Big)G(a_1(n),\ldots,a_m(n))=0,
\end{eqnarray*}
as was to be shown.
\end{proof}

\section{Total joint ergodicity for special classes of polynomials}\label{Sec:SCP}

In this section, we prove Theorems~\ref{T:TP}, \ref{cor1}, 
 \ref{cor2}, and \ref{cor3}.
We first provide in Subsection~\ref{ss:tjefs} variations of the last three theorems for a fixed system (as we did in Proposition~\ref{thm1_TP} for Theorem~\ref{T:TP}), and then derive the desired results in Subsection~\ref{sss:32}.

\subsection{Total joint ergodicity for a fixed  totally ergodic system}\label{ss:tjefs}

 We start with an implication of Proposition~\ref{thm1_TP}. Restricting to $\mathcal{P},$ we get the following.
 
\begin{proposition}\label{thm1}
Let $\ell\in \mathbb{N},$ $(X,\mathcal{B},\mu,T)$ be a totally ergodic system, and $p_{1},\dots,p_{\ell}\in\mathbb{R}[x]$. 
If $p_{1},\dots,p_{\ell}$ are $S(T)\backslash\mathbb{Z}_{\ast}$-independent,
then $([p_{1}(n)])_n,\dots,([p_{\ell}(n)])_n$ are totally jointly ergodic  for $(X,\mathcal{B},\mu,T)$. 
\end{proposition}

For $W\in \mathbb{N},$ we now define the notion of $W$-joint ergodicity.

\begin{definition*}
For $\ell\in\mathbb{N},$ a collection of sequences $a_1,\ldots,a_\ell:\mathbb{N}\to \mathbb{Z},$ and $W\in \mathbb{N},$ we say that $(a_1(n))_n,\ldots,$ $(a_\ell(n))_n$ are \emph{$W$-jointly ergodic} for the system $(X,\mathcal{B},\mu,T),$ if for all $f_{1},\dots,f_{\ell}\in L^{\infty}(\mu)$, we have 
\[\lim_{N\to\infty}\frac{1}{N}\sum_{n=1}^{N}T^{a_{1}(Wn)}f_{1}\cdot\ldots\cdot T^{a_{\ell}(Wn)}f_{\ell}=\int_{X}f_{1}\,d\mu\cdot\ldots\cdot \int_{X}f_{\ell}\,d\mu,\]
where the convergence takes place in $L^2(\mu).$ For $\ell=1,$ we say that $(a_1(n))_n$ is \emph{$W$-ergodic}.
\end{definition*} 
 
\begin{proposition}\label{thm2}
Let $\ell\in\mathbb{N},$ $(X,\mathcal{B},\mu,T)$ be a totally ergodic system, and $p_{1},\dots,p_{\ell}\in\mathbb{R}[x]$. 
If there exists a nonempty subset $\{i_{1},\dots,i_{k}\}\subseteq\{1,\dots,\ell\}$ such that $p_{i_{1}},\ldots,p_{i_{k}}$ are $S(T)\backslash\mathbb{Z}_{\ast}$-dependent and all the irrational polynomials in $p_{i_{1}},\dots,p_{i_{k}},$ if any, are $\mathbb{Q}$-independent,
then there exists $W_{0}\equiv W_0(p_1,\ldots,p_\ell) \in\mathbb{N}$ such that for all $W\geq W_{0}$, $([p_{1}(n)])_n,\dots,$ $([p_{\ell}(n)])_n$ are not $W!$-jointly ergodic for $(X,\mathcal{B},\mu,T)$. 
\end{proposition}

\begin{proof}    By setting the functions $f_{i}$ to be constant 1 for polynomials outside of the set $\{p_{i_1},\ldots,p_{i_k}\}$,  we may assume without loss of generality that $\{i_{1},\dots,i_{k}\}=\{1,\dots,\ell\}$.
Then, there exist  $c_{i}\in S(T)\backslash\mathbb{Z}_{\ast},$ not all of them $0$ with \begin{equation}\label{E:tje}
c_{1}p_{1}+\dots+c_{\ell}p_{\ell}=q\in\mathbb{Q}[x]+\mathbb{R}.\end{equation} 
We may assume without loss of generality that $p_{1},\dots,p_{m}\notin\mathbb{Q}[x]+\mathbb{R}$   and $p_{m+1},\dots,p_{\ell}\in\mathbb{Q}[x]+\mathbb{R}$. By the assumption, we have that $p_{1},\dots,p_{m}$ are $\mathbb{Q}$-independent.
Multiplying both sides of \eqref{E:tje} by an integer if necessary, we may assume without loss of generality that $q\in\mathbb{Z}[x]+\mathbb{R}$. It is clear that
there exists $W_{0}\equiv W_0(p_1,\ldots,p_\ell) \in\mathbb{N}$
	such that for all $W\geq W_{0}$, we have that $p_{i}(W!n)\in\mathbb{Z}[n]+\mathbb{R},$ $m+1\leq i\leq \ell$. 
	It suffices to show that 
	\begin{equation}\nonumber
	\begin{split}
	\lim_{N\to\infty}\frac{1}{N}\sum_{n=1}^{N}e\Bigl(\sum_{i=1}^{\ell}c_{i}[p_{i}(W!n)]\Bigr)\neq 0.
	\end{split}
	\end{equation}
Note that the left-hand side of the previous relation is equal to
\[\lim_{N\to\infty}\frac{1}{N}\sum_{n=1}^{N}e\Bigl(q(W!n)-\sum_{i=1}^{\ell}c_{i}\{p_{i}(W!n)\}\Bigr) = \lim_{N\to\infty}\frac{1}{N}\sum_{n=1}^{N}e\Bigl(-\sum_{i=1}^{m}c_{i}\{p_{i}(W!n)\}\Bigr)\cdot v,\]
	where $v=e\left(q(0)-\sum_{i=m+1}^{\ell}c_{i}\{p_{i}(0)\}\right)$.
	
	Let $F\colon\mathbb{T}^{m}\to\mathbb{C},$ with $F(x_{1},\dots,x_{m})=e\Bigl(-\sum_{i=1}^{m}c_{i}x_{i}\Bigr).$ Then $F$ is a Riemann integrable function.  Since $p_{1},\dots,p_{m}$ are $\mathbb{Q}$-independent, the same is true for $p_{1}(W!\cdot),\dots,$ $p_{m}(W!\cdot)$. By Weyl's criterion, $(\{p_{1}(W!n)\},\dots,\{p_{m}(W!n)\})_{n}$ is equidistributed on $\mathbb{T}^{m}$. So, 
	\begin{eqnarray*}\label{tempeq1}
	\lim_{N\to\infty}\frac{1}{N}\sum_{n=1}^{N}e\Bigl(-\sum_{i=1}^{m}c_{i}\{p_{i}(W!n)\}\Bigr)
	& = &\int_{[0,1]^{m}}F(x_{1},\dots,x_{m})\,d(x_{1},\dots,x_{m})
	\\&= &\int_{[0,1]^{m}}e\Bigl(-\sum_{i=1}^{m}c_{i}x_{i}\Bigr)\,d(x_{1},\dots,x_{m})
	\\&=&\prod_{i=1}^{m}\int_{[0,1]}e(-c_{i}x)\,dx.
	\end{eqnarray*}
	For all $1\leq i\leq m$,
	since $c_{i}\notin\mathbb{Z}_\ast$, we have that $\int_{[0,1]}e(-c_{i}x)\,dx=1$ if $c_{i}=0$ and $\int_{[0,1]}e(-c_{i}x)\,dx=\frac{1-e(-c_{i})}{2\pi i c_{i}}\neq 0$ if $c_{i}\notin\mathbb{Z}$;
	 the proof is complete.
\end{proof}

Combining the previous two results, we get the following result.

\begin{corollary}\label{thm3}
Let $\ell\in\mathbb{N},$ $(X,\mathcal{B},\mu,T)$ be a totally ergodic system, and $p_{1},\dots,p_{\ell}\in\mathbb{Q}[x]+\mathbb{R}$ (resp. $p_{1},\dots,p_{\ell}\in\mathbb{R}[x]$ so that all the irrational polynomials in $p_1,\dots,p_\ell,$ if any, are $\mathbb{Q}$-independent). Then the following are equivalent:
\begin{enumerate}[(i)]
\item[$(i)$] $([p_{1}(n)])_n,\ldots,([p_{\ell}(n)])_n$ are totally jointly ergodic for $(X,\mathcal{B},\mu,T)$.
\item[$(ii)$] There exists $W_{0}\equiv W_0(p_1,\ldots,p_\ell)\in\mathbb{N}$ such that $([p_{1}(n)])_n,\dots,([p_{\ell}(n)])_n$ are $W!$-jointly ergodic for $(X,\mathcal{B},\mu,T)$ for all $W\geq W_{0}$.
\item[$(iii)$] There exists an infinite set $I\equiv I(p_1,\ldots,p_\ell)\subseteq\mathbb{N}$ such that $([p_{1}(n)])_n,\dots,([p_{\ell}(n)])_n$ are $W!$-jointly ergodic for $(X,\mathcal{B},\mu,T)$ for all $W\in I$.
\item[$(iv)$] 	$p_{1},\dots,p_{\ell}$ are $S(T)$-independent (resp. $S(T)\backslash\mathbb{Z}_{\ast}$-independent).
\end{enumerate}
\end{corollary}

\begin{proof}
The implications $(i)\Rightarrow (ii)\Rightarrow (iii)$ are immediate, while the implications $(iii)\Rightarrow (iv)$ and $(iv)\Rightarrow (i)$ follow from Proposition~\ref{thm2} and Proposition~\ref{thm1} respectively. (When $p_{1},\dots,p_{\ell}\in\mathbb{Q}[x]+\mathbb{R},$ $S(T)\backslash\mathbb{Z}_{\ast}$-independence is equivalent to $S(T)$-independence.)
\end{proof}

\begin{remark}
By setting $\ell=1$ to the previous result, we get a characterization of when $([p(n)])_n,$ where $p\in \mathbb{R}[x],$ is totally ergodic for a specific totally ergodic system. 
\end{remark}

Notice that Corollary~\ref{thm3} also implies the result of Frantzikinakis and Kra (\cite[Theorem~1.1]{FK2}) which was presented in the introduction (in which joint ergodicity and total joint ergodicity are equivalent). Indeed, if $\{p_1,\ldots,p_\ell\}$ are rationally independent integer polynomials (in which case of course they belong to $\mathbb{Q}[x]+\mathbb{R}$), then they are $S(T)$-independent (since $S(T)$ consists of irrational numbers and $0$). So, by the previous corollary, we have total joint ergodicity, hence joint ergodicity.

\subsection{Total joint ergodicity for all totally ergodic systems}\label{sss:32}

We are now ready to use results from Subsections~\ref{SSJE} and \ref{ss:tjefs} to prove Theorems~\ref{cor1}, ~\ref{cor2}, ~\ref{cor3} and ~\ref{T:TP}; we start with Theorem~\ref{T:TP}.
 
\begin{proof}[Proof of Theorem~\ref{T:TP}] Let  $(X,\mathcal{B},\mu,T)$ be a totally ergodic system. Since every non-trivial linear combination
of $a_1,\ldots,a_\ell,$ with coefficients from $\mathbb{R}\backslash\mathbb{Q}_{\ast},$ stays logarithmically away from rational polynomials, the same is true for coefficients from $S(T)\backslash\mathbb{Z}_{\ast}$. By Proposition~\ref{thm1_TP}, $([a_{1}(n)])_n,\ldots,([a_{\ell}(n)])_n$ are totally jointly ergodic for $(X,\mathcal{B},\mu,T)$.
\end{proof}

Immediate implication of the previous result is Theorem~\ref{cor1}.

\begin{proof}[Proof of Theorem~\ref{cor1}]
The result follows from Theorem~\ref{T:TP} since when we restrict to $\mathcal{P},$ ``stays logarithmically away from rational polynomials'' is equivalent to not be in $\mathbb{Q}[x]+\mathbb{R}.$
\end{proof}
	
To prove Theorem~\ref{cor2}, we need the following lemma.

\begin{lemma}\label{construct}
For any $\ell\in \mathbb{N}$ and $c_{1},\dots,c_{\ell}\in\mathbb{R}$, there exist $D\in\mathbb{N}$ and a totally ergodic system $(X,\mathcal{B},\mu,T)$ such that $Dc_{1},\dots,Dc_{\ell}\in S(T)$. In particular, if $c_{1},\dots,c_{\ell}$ and $1$ are $\mathbb{Q}$-independent, we can choose $T$ so that $S(T)=\text{\emph{span}}_{\mathbb{Z}}\{1,c_{1},\dots,c_{\ell}\}$.
\end{lemma}

\begin{proof}
If all the $c_{i}$'s are rational numbers, then we may take $X$ to be the trivial system and $D$ to be the product of the denominators of the $c_{i}$'s. So, we assume that at least one of the $c_{i}$'s is irrational. 
It is not hard to see that there exist $a_{1},\dots,a_{m}\in\mathbb{R}$ for some $1\leq m\leq \ell$ such that $a_{1},\dots,a_{m},1$ are $\mathbb{Q}$-independent,  and that each of the $c_{1},\dots,c_{\ell}$ is a linear combination of $a_{1},\dots,a_{m},1$ with rational coefficients. So, there exists $D\in\mathbb{N}$ such that each of $Dc_{1},\dots,Dc_{\ell}$ is a linear combination of $a_{1},\dots,a_{m},1$ with integer coefficients. Hence, it suffices to construct a totally ergodic system %$(X,\mathcal{B},\mu,T)$
with $\{a_{1}\},\dots,\{a_{m}\}\in \text{Spec}(T)$.

Let $X=\mathbb{T}^{m}$ be endowed with the Haar measure and let $T\colon X\to X$, with \[T(x_{1},\dots,x_{m})=(x_{1}+a_{1},\dots,x_{m}+a_{m}).\] We have that $\{a_{1}\},\dots,\{a_{m}\}\in \text{Spec}(T)$. Moreover, since $a_{1},\dots,a_{m},1$ are $\mathbb{Q}$-independent, it is not hard to see that this system is totally ergodic.

The ``in particular'' part is straightforward from the proof (in which case $\{c_{1},\dots,c_{\ell}\}=\{a_{1},\dots,a_{m}\}$).
\end{proof}
	
\begin{proof}[Proof of Theorem~\ref{cor2}] Let $\{i_{1},\dots,i_{k}\}\subseteq\{1,\dots,\ell\}$ be such that $p_{i_{1}},\dots,p_{i_{k}}$ are not $\mathbb{R}\backslash\mathbb{Q}_{\ast}$-independent and all the irrational polynomials in $p_{i_{1}},\dots,p_{i_{k}}$ (if any) are $\mathbb{Q}$-independent. Assume that $c_{1}p_{i_{1}}+\dots+c_{k}p_{i_{k}}\in\mathbb{Q}[x]+\mathbb{R}$ for some $c_{1},\dots,c_{k}\in \mathbb{R}\backslash\mathbb{Q}_{\ast}$. 
By Lemma~\ref{construct}, there exists a totally ergodic system  $(X,\mathcal{B},\mu,T)$ and $D\in\mathbb{N}$ such that $Dc_{1},\dots,Dc_{k}\in S(T)$. Then $Dc_{1}p_{i_{1}}+\dots+Dc_{k}p_{i_{k}}\in\mathbb{Q}[x]+\mathbb{R}.$ Since $c_{1},\dots,c_{k}\notin\mathbb{Q}_{\ast}$, we have that $Dc_{1},\dots,Dc_{k}\in S(T)\backslash\mathbb{Z}_{\ast}$. 
By Proposition \ref{thm2}, there exists $W_{0}\in\mathbb{N}$ depending only on $p_{1},\dots,p_{\ell}$ such that  $([p_{1}(n)])_n,\ldots, ([p_{\ell}(n)])_n$ are not $W!$-jointly ergodic for $(X,\mathcal{B},\mu,T)$ for all $W\geq W_{0}$, from where the result follows. 
\end{proof}

The following is a corollary of Theorems~\ref{cor1} and ~\ref{cor2}.

\begin{corollary}\label{cor5.6}
 For $\ell\in\mathbb{N},$ let $p_{1},\dots,p_{\ell}\in\mathbb{Q}[x]+\mathbb{R}$ (resp. $p_{1},\dots,p_{\ell}\in\mathbb{R}[x]$ so that all the irrational polynomials in $p_1,\dots,p_\ell,$ if any, are $\mathbb{Q}$-independent). Then the following are equivalent:
 \begin{enumerate}[(i)]
 \item[$(i)$] $([p_{1}(n)])_n,\ldots,([p_{\ell}(n)])_n$ are totally jointly ergodic for every totally ergodic system.
 \item[$(ii)$]  There exists $W_{0}\equiv W_0(p_1,\ldots,p_\ell)\in\mathbb{N}$ such that for any $W\geq W_{0},$ $([p_{1}(n)])_n,\ldots,$ $([p_{\ell}(n)])_n$ are $W!$-jointly ergodic for every totally ergodic system.
 \item[$(iii)$] There exists an infinite set $I\equiv I(p_1,\ldots,p_\ell)\subseteq\mathbb{N}$ so that for any $W\in I,$ $([p_{1}(n)])_n,\ldots,$ $([p_{\ell}(n)])_n$ are $W!$-jointly ergodic for every totally ergodic system.
 \item[$(iv)$] $p_{1},\dots,p_{\ell}$ are $\mathbb{R}\backslash\mathbb{Q}_{\ast}$-independent.
 \end{enumerate}
\end{corollary}

\begin{proof}
It is clear that (i)$\Rightarrow$(ii)$\Rightarrow$(iii). 
The implication (iii)$\Rightarrow$(iv) follows from the proof of Theorem~\ref{cor2}.
Finally, the implication (iv)$\Rightarrow$(i) follows from Theorem~\ref{cor1}.
\end{proof}

\begin{proof}[Proof of Theorem~\ref{cor3}]
Follows immediately by the previous corollary.
\end{proof}

By letting $\ell=1$ in Theorem~\ref{cor3}, we get that Conjecture~\ref{Conjecture:1} holds for a single real polynomial.

\begin{corollary}
\label{cor55}
Let $p\in\mathbb{R}[x]$. $p$ is $\mathbb{R}\backslash\mathbb{Q}_{\ast}$-independent\footnote{ I.e., $p$ cannot be written as $p=cq,$ where $q\in\mathbb{Q}[x]+\mathbb{R}$ and $c\in \mathbb{R}\backslash\mathbb{Q}_{\ast}$.}  if, and only if, $([p(n)])_n$ is totally ergodic for every totally ergodic system.
\end{corollary}

While for $p_{1},\dots,p_{\ell}\in\mathbb{Q}[x]+\mathbb{R}$  (resp. $p_{1},\dots,p_{\ell}\in\mathbb{R}[x]$ so that all the irrational polynomials in $p_1,\dots,p_\ell,$ if any, are $\mathbb{Q}$-independent) we have that $W!$-joint ergodicity for infinitely many $W$'s is equivalent to total joint ergodicity (result that holds for a single arbitrary $p\in \mathbb{R}[x]$ as well), for two general polynomial iterates $p_1, p_2\in \mathbb{R}[x]$ we can have that $W!$-joint ergodicity for a co-finite set doesn't even imply joint ergodicity (see Proposition~\ref{apd_2}).

\section{Integrals of exponential functions on subtori}\label{sec:4}

In this section we will prove two statements, Propositions~\ref{lem:001} and ~\ref{lem:002}, that will help us deal with equidistribution results on subtori.  In particular, the integrals that we are computing in these two statements have connection with exponential sums, and will be used in the proof of our main result, Theorem~\ref{cor6}, later in the article. 

We first introduce some helpful notation: For $(x,y)\in\mathbb{R}^2\backslash\{(x,0)\colon x\in \mathbb{R}_\ast\}$ we let 
\[\left(\frac{x}{y}\right)_{\ast}:= \begin{cases} 
      x/y, & y\neq 0 \\
      1, & x=y=0 
   \end{cases}.\]
    In particular, for $(x,y)\in\mathbb{R}^2\backslash\{(x,0)\colon x\in \mathbb{R}_\ast\},$ 
   \[\left(\frac{x}{y}\right)_{\ast}\neq 0\;\Leftrightarrow\; x\neq 0 \text{ or } y=0.\]
Recall that $e(x)=e^{2\pi i x},$ $x\in\mathbb{R}.$ Using the previous notation, skipping the trivial computations, we have the following lemma.
 
\begin{lemma}\label{L:notation}
For all $\alpha\in\mathbb{R},$ $N\in\mathbb{N},$ and $t\in\mathbb{R}_\ast$, we have that 
    \begin{equation}\label{L121}
        \begin{split}
            \frac{1}{N}\sum_{n=0}^{N-1}e(\alpha n)=\left(\frac{e(\alpha N)-1}{N(e(\alpha)-1)}\right)_{\ast},
        \end{split}
    \end{equation}
and
\begin{equation}\label{L122}
        \begin{split}
            \frac{1}{t}\int_{0}^{t}e(\alpha x)\,dx=\left(\frac{e(\alpha t)-1}{2\pi i \alpha t}\right)_{\ast}.
        \end{split}
    \end{equation}
In particular, \eqref{L121} is nonzero if, and only if, $\alpha\notin (\mathbb{Z}/N)\backslash\mathbb{Z}$, and \eqref{L122} is nonzero if, and only if, $\alpha\notin (\mathbb{Z}/t)_\ast,$ where, for $s\in \mathbb{R}_\ast,$ we set $\mathbb{Z}/s:=\{a\in \mathbb{R}:\;as\in \mathbb{Z}\}.$ 
\end{lemma}

\medskip

In the rest of the paper, we will use Lemma \ref{L:notation} freely without citations. We start with an estimate for the integral of exponential functions along one dimensional subtori of $\mathbb{T}^{2}$.

\medskip

\begin{proposition}\label{lem:001}
Let $\alpha,\beta\in\mathbb{R},$ $a,b\in\mathbb{Z}_\ast$ with $\emph{\text{gcd}}(a,b)=1,$ $Y=\big\{(\{x\},\{y\})\colon (x,y)\in\mathbb{R}^{2},$ $ ax+by=0\big\},$ and $m_{Y}$ be the Haar measure on $Y.$ Then 
\[\int_{Y}e\big(\alpha x+\beta y\big)\,dm_{Y}(x,y)\neq 0\]
if, and only if, the following conditions hold:	$\frac{\alpha}{a}\notin (\mathbb{Z}/a)\backslash\mathbb{Z},\; \frac{\beta}{b}\notin (\mathbb{Z}/b)\backslash\mathbb{Z}, \frac{\alpha}{a}-\frac{\beta}{b}\notin \mathbb{Z}_\ast.$
\end{proposition}

\begin{proof}
Let $a, b$ have the same sign, say $a,b>0$.\footnote{ The case where $a,b$ have different signs is very similar and gives the same result modulo a nonzero constant.} Since $ax+by=0$, we have $x=bt$, and $y=-at,$ for some $t\in \mathbb{R}$. So 
\begin{eqnarray*}\nonumber
 \int_{Y}e\big(\alpha x+\beta y\big)\,dm_{Y}(x,y)
& = & \int_{0}^{1}e\big(\alpha \{bt\}+\beta \{-at\}\big)\,dt
\\&= &\sum_{j=0}^{ab-1}\int_{\frac{j}{ab}}^{\frac{j+1}{ab}}e\big(\alpha \{bt\}+\beta \{-at\}\big)\,dt
\\&=&\sum_{j=0}^{ab-1}\int_{0}^{\frac{1}{ab}}e\Bigl(\alpha \Big\{bt+\frac{j}{a}\Big\}+\beta \Big\{-at-\frac{j}{b}\Big\}\Bigr)\,dt
\\&=&\sum_{j=0}^{ab-1}\int_{0}^{\frac{1}{ab}}e\Bigl(\alpha \{bt\}+\beta \{-at\}+\alpha\Big\{\frac{j}{a}\Big\}+\beta\Big\{-\frac{j}{b}\Big\}-\beta\Bigr)\,dt
\\&=&\sum_{j=0}^{ab-1}\int_{0}^{\frac{1}{ab}}e\Bigl(\alpha bt-\beta at+\alpha\Big\{\frac{j}{a}\Big\}+\beta\Big\{-\frac{j}{b}\Big\}\Bigr)\,dt
\\&=&\sum_{j=0}^{ab-1}e\Bigl(\alpha\Big\{\frac{j}{a}\Big\}+\beta\Big\{-\frac{j}{b}\Big\}\Bigr)\cdot \frac{1}{ab} \left(\frac{e\Big(\frac{\alpha}{a}-\frac{\beta}{b}\Big)-1}{2\pi i \Big(\frac{\alpha}{a}-\frac{\beta}{b}\Big)}\right)_{\ast}.
\end{eqnarray*}
By the Chinese remainder theorem, $\Big(\Big\{\frac{j}{a}\Big\},\Big\{-\frac{j}{b}\Big\}\Big),$ $0\leq j\leq ab-1,$ takes all the values of $\Big(\Big\{\frac{m}{a}\Big\},$ $\Big\{-\frac{n}{b}\Big\}\Big),$ $0\leq m\leq a-1,$ $0\leq n\leq b-1$, and so it takes each value exactly once. 
Then
\begin{eqnarray*}
\sum_{j=0}^{ab-1}e\Bigl(\alpha\Big\{\frac{j}{a}\Big\}+\beta\Big\{-\frac{j}{b}\Big\}\Bigr) & = &\sum_{k=0}^{a-1}\sum_{j=0}^{b-1}e\Bigl(\alpha\Big\{\frac{k}{a}\Big\}\Bigr)\cdot e\Bigl(\beta\Big\{-\frac{j}{b}\Big\}\Bigr)
\\& = & e(\beta)\sum_{k=0}^{a-1}\sum_{j=0}^{b-1}e\Bigl(\alpha\cdot\frac{k}{a}\Bigr)\cdot e\Bigl(-\beta\cdot\frac{j}{b}\Bigr)
\\ & = & e(\beta)ab \left(\frac{e(\alpha)-1}{a\Big(e\Big(\frac{\alpha}{a}\Big)-1\Big)}\right)_{\ast}\cdot\left(\frac{e(-\beta)-1}{b\Big(e\Big(-\frac{\beta}{b}\Big)-1\Big)}\right)_{\ast}.
\end{eqnarray*}
To sum up, we have that
\[\int_{Y}e\big(\alpha x+\beta y\big)\,dm_{Y}(x,y)\neq 0\;\;\text{if, and only if,}\;\;\] \[\frac{\alpha}{a}\notin (\mathbb{Z}/a)\backslash\mathbb{Z},\; \frac{\beta}{b}\notin (\mathbb{Z}/b)\backslash\mathbb{Z}, \frac{\alpha}{a}-\frac{\beta}{b}\notin \mathbb{Z}_\ast,\]
as was to be shown.				
\end{proof}	

We next  estimate  the integral of exponential functions along one dimensional subtori of $\mathbb{T}^{3}$. Such estimates turn out to be very difficult to compute. We only provide an estimate for a special case which will be used in our approach.

\begin{proposition}\label{lem:002}
Let $\alpha,\beta\in\mathbb{R}$, $w\in\mathbb{Z},$  $a,b,r\in\mathbb{Z}_\ast$  with $\emph{\text{gcd}}(a,b)=1,$ $Y=\big\{(\{x\},\{y\},\{z\})\colon$  $(x,y,z)\in\mathbb{R}^{3}, ax+by=rx+bz=0\big\},$ and $m_{Y}$ be the Haar measure on $Y$. Then 
\begin{equation}\nonumber
\begin{split}
\int_{Y}e\big(\alpha x+\beta y+wz\big)\,dm_{Y}(x,y,z)\neq 0			
\end{split}
\end{equation}
if, and only if, the following conditions hold:
\[-\frac{\alpha}{a}+\frac{\beta}{b}+\frac{rw}{ab}\notin \mathbb{Z}_\ast, -\frac{\alpha}{a}+\frac{rwb^{\ast}}{a}\notin (\mathbb{Z}/a)\backslash\mathbb{Z}, \frac{\beta}{b}+\frac{rwa^{\ast}}{b}\notin (\mathbb{Z}/b)\backslash\mathbb{Z},\]  where $a^{\ast}$ is the unique element in $[1,b-1]$ such that $aa^{\ast}\equiv 1 \mod b$, and $b^{\ast}$ is the unique element in $[1,a-1]$ such that $bb^{\ast}\equiv 1 \mod a.$ 
\end{proposition}

\begin{proof}
Let $a, b, r>0.$\footnote{ We can assume that $r>0$ (this follows by the fact that $w\in \mathbb{Z}$); the case where $a,b$ have different signs is very similar and gives the same result modulo a nonzero constant.} Since $ax+by=rx+bz=0$, we may write $x=-bt,$ $y=at,$ and $z=rt,$ for some $t\in\mathbb{R}$.
	So, using the fact that $w$ is an integer, we get
		\begin{equation}\nonumber
			\begin{split}
			& \int_{Y}e\big(\alpha x+\beta y+wz\big)\,dm_{Y}(x,y,z)
			\\&=\int_{0}^{1}e\big(\alpha \{-bt\}+\beta \{at\}+w\{rt\}\big)\,dt
			\\&=\sum_{j=0}^{ab-1}\int_{\frac{j}{ab}}^{\frac{j+1}{ab}}e\big(\alpha \{-bt\}+\beta \{at\}+w\{rt\}\big)\,dt
			\\&=\sum_{j=0}^{ab-1}\int_{0}^{\frac{1}{ab}}e\Bigl(\alpha \Big\{-bt-\frac{j}{a}\Big\}+\beta \Big\{at+\frac{j}{b}\Big\}+w\Big\{rt+\frac{rj}{ab}\Big\}\Bigr)\,dt
			\\&=\sum_{j=0}^{ab-1}\int_{0}^{\frac{1}{ab}}e\Bigl(\alpha \{-bt\}+\beta \{at\}+w\{rt\}+\alpha \Big\{-\frac{j}{a}\Big\}+\beta \Big\{\frac{j}{b}\Big\}+w\Big\{\frac{rj}{ab}\Big\}-\alpha\Bigr)\,dt
			\\&=\sum_{j=0}^{ab-1}\int_{0}^{\frac{1}{ab}}e\Bigl((-\alpha b+\beta a+rw)t+\alpha \Big\{-\frac{j}{a}\Big\}+\beta \Big\{\frac{j}{b}\Big\}+w\Big\{\frac{rj}{ab}\Big\}\Bigr)\,dt
				\\&=\left(\frac{e\Big(-\frac{\alpha}{a}+\frac{\beta}{b}+\frac{rw}{ab}\Big)-1}{2\pi i\Big(-\frac{\alpha}{a}+\frac{\beta}{b}+\frac{rw}{ab}\Big)}\right)_{\ast}\cdot\frac{1}{ab}\sum_{j=0}^{ab-1}e\Bigl(\alpha \Big\{-\frac{j}{a}\Big\}+\beta \Big\{\frac{j}{b}\Big\}+w\Big\{\frac{rj}{ab}\Big\}\Bigr).
			\end{split}
			\end{equation}
 As we saw in Proposition~\ref{lem:001}, for all $0\leq j\leq ab-1,$ there exist unique $0\leq m\leq a-1,$ $0\leq n\leq b-1,$ so that 
\[\Big(\Big\{-\frac{j}{a}\Big\},\Big\{\frac{j}{b}\Big\}\Big)=\Big(\Big\{-\frac{m}{a}\Big\},\Big\{\frac{n}{b}\Big\}\Big).\] 
Hence, $j$ is the unique integer in $[0,ab-1]$ such that $j\equiv m \mod a$ and $j\equiv n \mod b,$ and so
	\[j\equiv mbb^{\ast}+naa^{\ast}\mod ab,\]
	where $a^{\ast}$ is the unique element in $[1,b-1]$ such that $aa^{\ast}\equiv 1 \mod b$, and $b^{\ast}$ is the unique element in $[1,a-1]$ such that $bb^{\ast}\equiv 1 \mod a$. We have that
 \[\Big\{\frac{rj}{ab}\Big\}=\Big\{\frac{r(mbb^\ast+naa^\ast)}{ab}\Big\},\]
so
	\begin{equation}\nonumber
			\begin{split}
			& \frac{1}{ab}\sum_{j=0}^{ab-1}e\Bigl(\alpha \Big\{-\frac{j}{a}\Big\}+\beta \Big\{\frac{j}{b}\Big\}+w\Big\{\frac{rj}{ab}\Big\}\Bigr)
			\\&=\frac{1}{ab}\sum_{m=0}^{a-1}\sum_{n=0}^{b-1}e\Bigl(\alpha \Big\{-\frac{m}{a}\Big\}+\beta \Big\{\frac{n}{b}\Big\}+w\Big\{\frac{r(mbb^{\ast}+naa^{\ast})}{ab}\Big\}\Bigr) 
			\\&=\frac{1}{ab}\sum_{m=0}^{a-1}\sum_{n=0}^{b-1}e\Bigl(m\Big(-\frac{\alpha}{a}+\frac{rwb^{\ast}}{a}\Big)+n\Big(\frac{\beta}{b}+\frac{rwa^{\ast}}{b}\Big)-\alpha\Bigr)
			\\&=e(-\alpha)\left(\frac{e\Big(-\alpha+rwb^{\ast}\Big)-1}{a\Big(e\Big(-\frac{\alpha}{a}+\frac{rwb^{\ast}}{a}\Big)-1\Big)}\right)_{\ast}\cdot\left(\frac{e\Big(\beta+rwa^{\ast}\Big)-1}{b\Big(e\Big(\frac{\beta}{b}+\frac{rwa^{\ast}}{b}\Big)-1\Big)}\right)_{\ast}.
			\end{split}
			\end{equation}
Putting everything together, we have that 
\begin{equation}\nonumber
\begin{split}
			& \int_{Y}e\Bigl(\alpha x+\beta y+wz\Bigr)\,dm_{Y}(x,y,z)\neq 0\;\;\text{if, and only if,}
\end{split}
\end{equation}
\[-\frac{\alpha}{a}+\frac{\beta}{b}+\frac{rw}{ab}\notin \mathbb{Z}_\ast,  -\frac{\alpha}{a}+\frac{rwb^{\ast}}{a}\notin (\mathbb{Z}/a)\backslash\mathbb{Z}, \frac{\beta}{b}+\frac{rwa^{\ast}}{b}\notin (\mathbb{Z}/b)\backslash\mathbb{Z},\] 
as was to be shown.
\end{proof}

\section{Estimating averages of exponential sums}\label{Section:6}

\subsection{Properties of polynomial orbits}

For a fixed triple of real polynomials $(p_1,p_2,p_3),$ with $p_1(0)=p_2(0)=p_3(0)=0,$ we define a subtorus 
$Y=Y(p_1,p_2,p_3),$ so that the sequence $\big(\{p_1(W!n)\},\{p_2(W!n)\},\{p_3(W!n)\}\big)_n$ is equidistributed on $Y$ for large enough $W$ (see the definition before Proposition~\ref{green}).
  
 For such a triple of polynomials, let 
\[K(p_{1},p_{2},p_{3}):=\big\{(k_{1},k_{2},k_{3})\in\mathbb{Z}^{3}:\; k_{1}p_{1}+k_{2}p_{2}+k_{3}p_{3}\in\mathbb{Q}[x]\big\},\]
and
\begin{equation*}
    \begin{split}
        & Y(p_{1},p_{2},p_{3}):=\big\{(\{x\},\{y\},\{z\})\colon (x,y,z)\in\mathbb{R}^3, k_{1}x+k_{2}y+k_{3}z=0,
        \\
        & \quad \quad\quad\quad\quad\quad\quad\quad\quad\quad\quad\quad\quad\quad\quad\quad\quad\quad\quad\quad\quad\quad\quad\quad\quad (k_{1},k_{2},k_{3})\in K(p_{1},p_{2},p_{3})\big\}.
    \end{split}
\end{equation*}
We remark that, for all  $W\in\mathbb{N},$ we have $K(p_{1}(W\cdot),p_{2}(W\cdot),p_{3}(W\cdot))=K(p_{1},p_{2},p_{3})$  and $Y(p_{1}(W\cdot),p_{2}(W\cdot),p_{3}(W\cdot))=Y(p_{1},p_{2},p_{3})$.

\medskip

We have the following equidistribution property.

\begin{lemma}\label{abc_general2}
Let $p_{1},p_{2},p_{3}\in\mathbb{R}[x]$ be polynomials  with $p_{1}(0)=p_{2}(0)=p_{3}(0)=0$.
If  $\big(\{p_{1}(n)\},$ $\{p_{2}(n)\}, \{p_{3}(n)\}\big)_n$ takes values in $Y(p_{1},p_{2},p_{3})$, then it is also 
equidistributed on $Y(p_{1},p_{2},p_{3})$.
\end{lemma}
 
\begin{proof}
Suppose that this is not the case. Note that 
 $Y(p_{1},p_{2},p_{3})$  is a subtorus of $\mathbb{T}^3.$ 
 By Proposition \ref{green}, there exists $(k_{1},k_{2},k_{3})\in\mathbb{Z}^{3}$
 which does not belong to the $\mathbb{Q}$-span of $K(p_{1},p_{2},p_{3})$ such that $k_{1}p_{1}(n)+k_{2}p_{2}(n)+k_{3}p_{3}(n)\in\mathbb{Z}$ for all $n\in\mathbb{Z}$. However, this implies that $(k_{1},k_{2},k_{3})\in K(p_{1},p_{2},p_{3}),$ a contradiction by definition. Hence $\big(\{p_{1}(n)\}, \{p_{2}(n)\},$ $\{p_{3}(n)\}\big)_{n}$ is equidistributed on $Y(p_{1},p_{2},p_{3})$.
\end{proof}

\begin{remark}
A rather strong restriction in Lemma \ref{abc_general2} is that it is only applicable when $\big(\{p_{1}(n)\}, \{p_{2}(n)\}, \{p_{3}(n)\}\big)_n$   takes values in $Y(p_{1},p_{2},p_{3})$.
In general,  $\big(\{p_{1}(n)\}, \{p_{2}(n)\},$ $\{p_{3}(n)\}\big)_n$ takes values on finitely many shifted copies of $Y(p_{1},p_{2},p_{3})$. For example, if $p_{1}(n)=-\alpha n, p_{2}(n)=(\alpha+1/2)n$ and $p_{3}(n)=\beta n$ for some $\alpha,\beta\in\mathbb{R}$ with $\alpha,\beta,1$ being $\mathbb{Q}$-independent, then  
\begin{equation*}
    \begin{split}
         Y(p_{1},p_{2},p_{3}):=\big\{(\{x\},\{y\},\{z\})\colon (x,y,z)\in\mathbb{R}^3, x+y=0\big\}.
    \end{split}
\end{equation*}
However, $\big(\{p_{1}(n)\},$ $\{p_{2}(n)\}, \{p_{3}(n)\}\big)_n$ does not lie in $Y(p_{1},p_{2},p_{3})$, but in $Y(p_{1},p_{2},p_{3})\cup(Y(p_{1},p_{2},p_{3})+(0,1/2,0) \mod \mathbb{Z}^{3})$.
\end{remark}

To ensure that we can use Lemma \ref{abc_general2} for our purposes, we need another result (see Lemma~\ref{abc_general3} below). To state it, we need the following notation.

\medskip

Let $K_{0}$ be a basis of $K(p_{1},p_{2},p_{3})$. For each $(k_{1},k_{2},k_{3})\in K_{0}$, suppose that \[k_{1}p_{1}+k_{2}p_{2}+k_{3}p_{3}=g\] for some 
polynomial $g$ in $\mathbb{Q}[x]$.
If $Q(k_{1},k_{2},k_{3})$ denotes the smallest natural number such that
$g(Q(k_{1},k_{2},k_{3})!n)/\text{gcd}(k_{1},k_{2},k_{3})$ is integer-valued, we also let
\begin{equation}\label{defnw}
    W_{0}(p_{1},p_{2},p_{3}):=\min_{\substack{K_0\;\text{basis of} \\ K(p_{1},p_{2},p_{3})}}\max_{(k_{1},k_{2},k_{3})\in K_{0}}Q(k_{1},k_{2},k_{3}).
\end{equation}

\begin{lemma}\label{abc_general3}
Let $p_{1},p_{2},p_{3}\in\mathbb{R}[x]$ be polynomials with $p_{1}(0)=p_{2}(0)=p_{3}(0)=0$.
If $W\geq W_{0}(p_{1},p_{2},p_{3})$, where $W_0$ is given in \eqref{defnw},
then the sequence $\big(\{p_{1}(W!n)\},$ $\{p_{2}(W!n)\}, $ $ \{p_{3}(W!n)\}\big)_n$ takes values in $Y(p_{1},p_{2},p_{3}).$
\end{lemma}

\begin{proof}
Let $K_{0}$ be a basis of $K(p_{1},p_{2},p_{3})$ and 
  fix $(k_1,k_2,k_3)\in K_{0}$. 
  Then
  $k_1p_1+k_2p_2+k_3p_3=g\in \mathbb{Q}[x].$  For convenience denote $Q:=Q(k_{1},k_{2},k_{3})$.
  Let $a_{1},a_{2},a_{3}\in\mathbb{Z}$ be such that $k_{1}a_{1}+k_{2}a_{2}+k_{3}a_{3}=\text{gcd}(k_{1},k_{2},k_{3})$. Then $\{p_{i}(Q!n)\}=\left\{p_{i}(Q!n)-\frac{a_{i}g(Q!n)}{\text{gcd}(k_{1},k_{2},k_{3})}\right\}$ for $1\leq i\leq 3$. Moreover, 
  $$\sum_{i=1}^{3}k_{i}\left(p_{i}(Q!n)-\frac{a_{i}g(Q!n)}{\text{gcd}(k_{1},k_{2},k_{3})}\right)=g(Q!n)-g(Q!n)=0.$$
  So, $\big(\{p_1(Q!n)\}, \{p_2(Q!n)\}, \{p_3(Q!n)\}\big)$ 
 belongs to  $Y(p_{1},p_{2},p_{3})$. 
\end{proof}

As it was mentioned in previous sections, it is important for our study to consider the case of two rational polynomials where one is, or is not, a multiple of the other. In the latter case, we have the following helpful lemma. 

\begin{lemma}\label{c1}
Let $f,g\in\mathbb{Q}[x],$ $f(0)=g(0)=0,$ $g\not\equiv 0,$ and suppose that $f$ is not a multiple of $g$. If $af+bg\in \mathbb{Q}[x]$ for some $a,b\in\mathbb{R}$, then we must have that $a,b\in\mathbb{Q}$.
\end{lemma}	

\begin{proof}
Suppose that $f(n)=\sum_{i=1}^{K}f_{i}n^{i}$ and  $g(n)=\sum_{i=1}^{K}g_{i}n^{i},$ for some $K\in \mathbb{N}$ and $f_{i},g_{i}\in\mathbb{Q}$. Then $af_{i}+bg_{i}\in\mathbb{Q}$  for all $1\leq i\leq K$. If $b\notin\mathbb{Q}$, then,
since $g_{i'}\neq 0$ for some $1\leq i'\leq K$ and $f_{i'},g_{i'}\in\mathbb{Q},$ we must have that $a\notin\mathbb{Q}$. Since $a,b\notin \mathbb{Q}$ and $af_i+bg_i\in \mathbb{Q}$ for all $1\leq i\leq K$, we must also have that $f_ig_j-f_jg_i=0$ for all $1\leq i,j\leq K,$ which implies that $f$ is a multiple of $g$, a contradiction. Hence, $b\in\mathbb{Q},$ which in turn implies that $af\in \mathbb{Q}[x]$, so $a\in \mathbb{Q}$ too (since $f$, being not a multiple of $g$, is not constant zero).
\end{proof}	

\subsection{Characterizing when an exponential sum is zero}\label{Ss:main_comp}

The purpose of this subsection is to study the average
\begin{equation}\label{target}
    \lim_{N\to\infty}\frac{1}{N}\sum_{n=1}^{N}e\Bigl(\sum_{i=1}^{2}t_{i}[p_{i}(n)]\Bigr),
\end{equation}
where $t_{1},t_{2}\in\mathbb{R}$. For convenience denote by $p_{3}:=t_{1}p_{1}+t_{2}p_{2}$.
For two special types of polynomial sequences $(p_{1},p_{2},p_{3})$, 
we will compute the subtorus $Y(p_{1},p_{2},p_{3})$, and partially characterize when \eqref{target} is equal to $0.$

The first type of $(p_{1},p_{2},p_{3})$ we are interested in is given in the following proposition.

\begin{proposition}\label{P:Case 1}
 Let $f,g\in\mathbb{Q}[x],$ $f(0)=g(0)=0,$ $g\not\equiv 0,$ 
 $f$ is not a multiple of $g$, $c\in \mathbb{R}\backslash\mathbb{Q},$ $u_{1},u_{2}\in\mathbb{Z}_\ast$ with $\emph{\text{gcd}}(u_{1},u_{2})=1,$ and $d\in\mathbb{Q}_\ast.$ Let  
 \begin{equation}\nonumber
\begin{split}
p_{1}=u_{1}(f+cg),\;\; p_{2}=u_{2}(f+(c+d)g),
\end{split}
\end{equation}
 $t_{1},t_{2}\in (\mathbb{R}\backslash\mathbb{Q})\cup\mathbb{Z},$ not both in $\mathbb{Z},$ and $p_{3}=t_{1}p_{1}+t_{2}p_{2}$.
We have:
\begin{enumerate}[(i)]
    \item If $t_{1}u_{1}+t_{2}u_{2}\notin\mathbb{Q}$, then
    \begin{equation}\nonumber
      Y(p_{1},p_{2},p_{3})=\big\{(\{x\},\{y\},\{z\})\colon (x,y,z)\in\mathbb{R}^3, u_{2}x-u_{1}y=0\big\},
  \end{equation}
  and
    \eqref{target} is zero if  $\big(\{p_{1}(n)\}, \{p_{2}(n)\}, \{p_{3}(n)\}\big)_n$ takes values in $Y(p_{1},p_{2},p_{3})$;
    \item If $t_{1}u_{1}+t_{2}u_{2}\in\mathbb{Z}$, and $t_{2},c,1$ are $\mathbb{Q}$-independent, then
    \begin{equation}\nonumber
      Y(p_{1},p_{2},p_{3})=\big\{(\{x\},\{y\},\{z\})\colon (x,y,z)\in\mathbb{R}^3, u_{2}x-u_{1}y=0\big\},
  \end{equation}
  and
    \eqref{target} is zero if  $\big(\{p_{1}(n)\}, \{p_{2}(n)\}, \{p_{3}(n)\}\big)_n$ takes values in $Y(p_{1},p_{2},p_{3})$;
    \item If $t_{1}u_{1}+t_{2}u_{2}\in\mathbb{Z}$, and $t_{2}=ac+b$ for some $a,b\in\mathbb{Q}$, then  
    \begin{equation}\nonumber
      Y(p_{1},p_{2},p_{3})=\big\{(\{x\},\{y\},\{z\})\colon (x,y,z)\in\mathbb{R}^{3},u_{2}x-u_{1}y=au_2dx-u_{1}z=0\big\}.
  \end{equation}
  Moreover, if  $\big(\{p_{1}(n)\}, \{p_{2}(n)\}, \{p_{3}(n)\}\big)_n$ takes values in $Y(p_{1},p_{2},p_{3})$ and $au_{2}d\in\mathbb{Z},$ then
    \eqref{target} is nonzero if, and only if, $\frac{ad}{u_{1}}\notin\mathbb{Z}_{\ast}.$
\end{enumerate}
\end{proposition}

\begin{remark}\label{97}
As the corresponding $t_i$'s will be taken in $S(T),$ for a totally ergodic transformation, (i), (ii) and (iii) cover all possible cases since $S(T)\subseteq(\mathbb{R}\backslash\mathbb{Q})\cup\mathbb{Z}.$ 

Notice that in case (iii), even if $\big(\{p_{1}(n)\},$ $\{p_{2}(n)\}, \{p_{3}(n)\}\big)_n$ takes values in $Y(p_{1},p_{2},p_{3})$, we were only able to provide an estimate for \eqref{target} when $au_{2}d\in\mathbb{Z}.$ However, this is good enough for our purposes.
\end{remark}

\begin{proof}[Proof of Proposition~\ref{P:Case 1}]

For convenience, we denote by \[K:=K(p_{1},p_{2},p_{3}),\;\;\text{and}\;\; Y:=Y(p_{1},p_{2},p_{3}).\]
Note that \eqref{target} becomes
	\begin{equation}\nonumber
	\begin{split}
	&\lim_{N\to\infty}\frac{1}{N}\sum_{n=1}^{N}e\Bigl(\sum_{i=1}^{2}t_{i}[p_{i}(n)]\Bigr)
	\\&=\lim_{N\to\infty}\frac{1}{N}\sum_{n=1}^{N}e\Bigl(\Big\{\sum_{i=1}^{2}t_{i}p_{i}(n)\Big\}-\sum_{i=1}^{2}t_{i}\{p_{i}(n)\}\Bigr)
	\\&=\lim_{N\to\infty}\frac{1}{N}\sum_{n=1}^{N}e\Bigl(-t_{1}\{p_{1}(n)\}-t_{2}\{p_{2}(n)\}+\{p_{3}(n)\}\Bigr).
	\end{split}
	\end{equation}
	If $\big(\{p_{1}(n)\},$ $\{p_{2}(n)\}, \{p_{3}(n)\}\big)_n$ takes values in $Y$, then by Lemma \ref{abc_general2}, we have that
	\begin{equation}\label{E:NAY}
	\begin{split}
	 \lim_{N\to\infty}\frac{1}{N}\sum_{n=1}^{N}e\Bigl(\sum_{i=1}^{2}t_{i}[p_{i}(n)]\Bigr)
	=\int_{Y}e\big(-t_{1}x-t_{2}y+z\big)\,d \mu_{Y}(x,y,z).
	\end{split}
	\end{equation}
 
	In order to analyze the description of $Y$, we need to compute $K$. Recall that for $(k_{1},k_{2},k_{3})\in\mathbb{Q}^{3}$, we have that
 $(k_{1},k_{2},k_{3})\in K$ if
	\begin{equation}\label{18}
	\begin{split}
	&\quad k_{1}p_{1}+k_{2}p_{2}+k_{3}p_{3}= \big((k_{1}u_{1}+k_{2}u_{2}+k_{3}(t_{1}u_{1}+t_{2}u_{2}))c+k_{3}t_{2}u_{2}d+k_{2}u_{2}d\big)g	
	 \\& \quad\quad\quad\quad\quad\quad\quad\quad\quad\quad\quad \quad\quad\quad\quad\quad\quad+\big(k_{1}u_{1}+k_{2}u_{2}+k_3(t_1u_1+t_2u_2)\big)f \in\mathbb{Q}[n].
	\end{split}
	\end{equation}		
	Since $f$ is not a multiple of $g$,
then
	by Lemma \ref{c1}, (\ref{18}) holds if, and only if,
\begin{equation}\label{19}
k_3(t_1u_1+t_2u_2), (k_{1}u_{1}+k_{2}u_{2}+k_{3}(t_{1}u_{1}+t_{2}u_{2}))c+k_{3}t_{2}u_{2}d\in\mathbb{Q}.
\end{equation}

 We first prove part (i).
If 	$t_{1}u_{1}+t_{2}u_{2}\notin\mathbb{Q}$,   then
  (\ref{19}) holds if, and only if, $k_{3}=k_{1}u_{1}+k_{2}u_{2}=0$.
  In this case, $K$ is generated by the vector $(u_{2},-u_{1},0).$ Therefore,
  \begin{equation}\label{simpley}
      Y=\big\{(\{x\},\{y\},\{z\})\colon (x,y,z)\in\mathbb{R}^3, u_{2}x-u_{1}y=0\big\}.
  \end{equation}
Then \eqref{E:NAY} is 0 since $\int_{0}^{1}e(z)\,dz=0.$ This concludes the first part.

For part (ii), if $t_{1}u_{1}+t_{2}u_{2}\in\mathbb{Z}$, then 
we may replace $t_{i}$ by $t_{i}+m_{i}$ for some integer $m_{i}$ (notice that this does not change the sum or the conditions in parts (ii) or (iii)). Since $\text{gcd}(u_{1},u_{2})=1$, we may assume that $t_{1}u_{1}+t_{2}u_{2}=0$. 
	 (Indeed,  
  if $t_1u_1+t_2u_2=k,$ since $\text{gcd}(u_1,u_2)=1,$ by the Euclidean algorithm, there are integers $m_1, m_2$ such that $m_1u_1+m_2u_2=1,$ so, $km_1u_1+km_2u_2=k.$ Hence 
	$(t_1-km_1)u_1+(t_2-km_2)u_2=0.$)
 So, \eqref{19} becomes
 \begin{equation}\label{E:new_17}
     (k_1u_1+k_2u_2)c+k_3t_2u_2d\in \mathbb{Q}.
 \end{equation}

Suppose that $t_2, c, 1$ are $\mathbb{Q}$-independent. Since $u_{2}d\neq 0$,
  \eqref{E:new_17} holds if, and only if, $k_{3}=k_{1}u_{1}+k_{2}u_{2}=0$.
   Again in this case, $K$ is generated by the vector $(u_{2},-u_{1},0)$. As before, $Y$ is given by (\ref{simpley}), and so \eqref{E:NAY} is equal to 0. This concludes part (ii).

For part (iii), suppose that $t_2=ac+b,$ for some $a, b\in \mathbb{Q}$ with $au_2 d\in \mathbb{Z}.$
If $t_2\in \mathbb{Z},$ then, since $t_1u_1+t_2u_2=0$ and $t_{1}\in (\mathbb{R}\backslash\mathbb{Q})\cup\mathbb{Z},$ we have to have $t_1\in \mathbb{Z},$ a contradiction. Hence, we have that $t_2\notin\mathbb{Q}$  
(which also implies that $t_1\notin \mathbb{Q}$). Then, \eqref{E:new_17} is equivalent to  \[k_1u_1+k_2u_2+k_3au_2d=0.\]
 Therefore, $K$ is generated by the vectors $(u_{2},-u_{1},0)$ and $(au_{2}d,0,-u_{1})$. So,
\[Y=\big\{(\{x\},\{y\},\{z\})\colon (x,y,z)\in\mathbb{R}^{3},u_{2}x-u_{1}y=au_2dx-u_{1}z=0\big\}.\]
Since $au_{2}d\in\mathbb{Z}_{\ast}$,
setting $a=u_{2}, b=-u_{1}, r=au_{2}d, \alpha=-t_{1}, \beta=-t_{2}$ and $w=1$ in Proposition~\ref{lem:002}, we have that \eqref{E:NAY} is nonzero if, and only if, the following hold:
\begin{itemize}
    \item $\frac{ad}{u_{1}}\notin\mathbb{Z}_{\ast}$;
    \item $\frac{t_{1}}{u_{2}}+(-u_{1})^{\ast}ad\notin (\mathbb{Z}/u_{2})\backslash \mathbb{Z}$;
    \item $\frac{t_{2}}{u_{1}}-\frac{u_{2}^{\ast}au_{2}d}{u_{1}}\notin (\mathbb{Z}/u_{1})\backslash \mathbb{Z}$.
\end{itemize}   

Since $t_{1},t_{2}\notin\mathbb{Q}$, the last two conditions are always true. So, the average is nonzero if, and only if, $\frac{ad}{u_1}\notin\mathbb{Z}_{\ast}.$
This proves part (iii) and completes the proof.
\end{proof}

The second type of $(p_{1},p_{2},p_{3})$ we are interested in is given in the following proposition.

\begin{proposition}\label{P:Case 2}
 Let $g\in\mathbb{Q}[x],$ $g(0)=0,$ $g\not\equiv 0,$ 
     $c\in \mathbb{R}\backslash\mathbb{Q},$ $u_{1},u_{2}\in\mathbb{Z}_\ast$ with $\emph{\text{gcd}}(u_{1},u_{2})=1,$ and $d\in\mathbb{Q}_\ast.$  Let  
 \begin{equation}\nonumber
p_1=u_1cg,\;\;\;p_2=u_2(c+d)g,
\end{equation}
 $t_{1},t_{2}\in\mathbb{R},$ not both in $\mathbb{Z},$ and $p_{3}=t_{1}p_{1}+t_{2}p_{2}.$
We have:
\begin{enumerate}[(i)]
    \item   If $t_{1}u_{1}c+t_{2}u_{2}(c+d),c$ and 1 are $\mathbb{Q}$-independent, then
    \begin{equation}\nonumber
      Y(p_{1},p_{2},p_{3})=\big\{(\{x\},\{y\},\{z\})\colon (x,y,z)\in\mathbb{R}^3, u_{2}x-u_{1}y=0\big\}.
  \end{equation}
  Moreover, if $(\{p_{1}(n)\},\{p_{2}(n)\}, \{p_{3}(n)\})_n$ takes values in $Y(p_{1},p_{2},p_{3})$, then
   \eqref{target} is zero;
    \item  If $t_{1}u_{1}c+t_{2}u_{2}(c+d)=sc+t$
for some $s,t\in\mathbb{Q}$,   then
    \begin{equation}\nonumber
      Y(p_{1},p_{2},p_{3})=\big\{(\{x\},\{y\},\{z\})\colon (x,y,z)\in\mathbb{R}^3, u_{1}y-u_{2}x=u_{1}z-sx=0\big\}.
  \end{equation}
  Moreover, if $\big(\{p_{1}(n)\},\{p_{2}(n)\}, \{p_{3}(n)\}\big)_n$ takes values in $Y(p_{1},p_{2},p_{3})$ and $s\in\mathbb{Z},$ then
    \eqref{target} is nonzero if, and only if, the following hold: 
\begin{itemize}
    \item $\frac{t-t_{2}u_{2}d}{u_{1}u_{2}c}\notin \mathbb{Z}_\ast$;
    \item $\frac{t_{1}}{u_{2}}+\frac{s(-u_{1})^{\ast}}{u_{2}}\notin (\mathbb{Z}/u_{2})\backslash\mathbb{Z}$;
    \item $\frac{t_{2}}{u_{1}}-\frac{su_{2}^{\ast}}{u_{1}}\notin (\mathbb{Z}/u_{1})\backslash\mathbb{Z}$.
\end{itemize}
\end{enumerate}
\end{proposition}

\begin{remark}\label{98}
 As with Proposition~\ref{P:Case 1} (iii), in case (ii) of the previous statement,  we were only able to provide an estimate for \eqref{target} for the special case where $s\in\mathbb{Z}$, which suffices for our purposes.
\end{remark}

\begin{proof}[Proof of Proposition~\ref{P:Case 2}]
For convenience, we denote by \[K:=K(p_{1},p_{2},p_{3}),\;\;\text{and}\;\; Y:=Y(p_{1},p_{2},p_{3}).\]
Similar to the proof of Proposition~\ref{P:Case 1}, if $\big(\{p_{1}(n)\},$ $\{p_{2}(n)\}, \{p_{3}(n)\}\big)_n$ takes values in $Y$, then
 \begin{equation}\label{E:NAY_2}
	\begin{split}
	 \lim_{N\to\infty}\frac{1}{N}\sum_{n=1}^{N}e\Bigl(\sum_{i=1}^{2}t_{i}[p_{i}(n)]\Bigr)
	=\int_{Y}e\big(-t_{1}x-t_{2}y+z\big)\,d \mu_{Y}(x,y,z).
	\end{split}
	\end{equation}
 
	In order to analyze the description of $Y$, we need to compute $K$.
 Note that for $(k_{1},k_{2},k_{3})\in \mathbb{Q}^{3}$, we have $(k_{1},k_{2},k_{3})\in K$ if, and only if,
\begin{equation}\label{qwer_1}
\begin{split}
     (k_{1}u_{1}+k_{2}u_{2})c+k_{3}(t_{1}u_{1}c+t_{2}u_{2}(c+d))
    \in\mathbb{Q}.
\end{split}    
\end{equation}

We first prove part (i).
Since $t_{1}u_{1}c+t_{2}u_{2}(c+d),c$ and 1 are $\mathbb{Q}$-independent, we have that \eqref{qwer_1} holds if, and only if,
\[k_3=k_{1}u_{1}+k_{2}u_{2}=0.\]
Therefore, $K$ is generated by the vector $(u_{2},-u_{1},0)$. So,
$Y$ is the same as case (i) of Proposition~\ref{P:Case 1}, and (\ref{target}) equals $0$. This concludes part (i).

We next prove part (ii).
Since $c$ is irrational, \eqref{qwer_1} holds if, and only if, \[k_{1}u_{1}+k_{2}u_{2}+k_{3}s=0.\] Therefore, $K$ is generated by the vectors $(u_{2},-u_{1},0)$ and $(s,0,-u_{1})$. So, we have that
\[Y(p_{1},p_{2},p_{3})=\big\{(\{x\},\{y\},\{z\})\colon (x,y,z)\in\mathbb{R}^3, u_{1}y-u_{2}x=u_{1}z-sx=0\big\}.\]
By Proposition~\ref{lem:002}  for $a=-u_2, b=u_1, r=-s, \alpha=-t_1, \beta=-t_2,$ and $w=1,$ \eqref{E:NAY_2} is nonzero if, and only if, the following hold:
\begin{itemize}
    \item $\frac{t_{1}}{u_{2}}+\frac{t_{2}}{u_{1}}-\frac{s}{u_{1}u_{2}}=\frac{t-t_{2}u_{2}d}{u_{1}u_{2}c}\notin \mathbb{Z}_\ast$;
    \item $\frac{t_{1}}{u_{2}}+\frac{s(-u_{1})^{\ast}}{u_{2}}\notin (\mathbb{Z}/u_{2})\backslash\mathbb{Z}$;
    \item $\frac{t_{2}}{u_{1}}-\frac{su_{2}^{\ast}}{u_{1}}\notin (\mathbb{Z}/u_{1})\backslash\mathbb{Z}$.
\end{itemize}
This completes the proof of the statement.
\end{proof}

\section{Characterizing total joint ergodicity for two terms}\label{sec:7}

In this section, we prove our main result on polynomial iterates, i.e., Theorem~\ref{cor6}.

\subsection{Type-B pairs}\label{Ss:TB}
Our first step is to reduce the total joint ergodicity problem to the case of some special pairs which are closely related to the pairs appearing in
Theorem \ref{cor6} case (ii), which we refer as ``Type-B''. 

\begin{definition*}\label{D:TB}
Let $(X,\mathcal{B},\mu,T)$ be a system.
We say that the pair $(p_{1},p_{2})$ of real polynomials is of \emph{Type-B for the system $(X,\mathcal{B},\mu,T)$} if
 there exist $f,g\in\mathbb{Q}[x],$ $f(0)=g(0)=0,$ $g\not\equiv 0,$ $c\in S(T)\backslash\mathbb{Z},$ $u_{1},u_{2}\in\mathbb{Z}_\ast$ with $\text{gcd}(u_{1},u_{2})=1,$ and $d\in\mathbb{Q}_\ast,$ such that  
\begin{equation}\label{equ:assumption2}
\begin{split}
p_{1}=u_{1}(f+cg),\;\; p_{2}=u_{2}(f+(c+d)g).
\end{split}
\end{equation}
\end{definition*}

 \begin{proposition}\label{case2}
Let $(X,\mathcal{B},\mu,T)$ be a totally ergodic system.
    If  $p_{1},p_{2}$ are $S(T)\backslash\mathbb{Z}_{\ast}$-dependent with $p_{1}(0)=p_{2}(0)=0$, then either
    \begin{enumerate}
    \item[$(i)$] there exists $W_0\equiv W_{0}(p_{1},p_{2})\in\mathbb{N}$  such that $([p_{1}(n)])_n, ([p_{2}(n)])_n$ are not $W!$-jointly ergodic for $(X,\mathcal{B},\mu,T)$ for all $W\geq W_{0}$; or
    \item[$(ii)$] $(p_1,p_2)$ is of Type-B for $(X,\mathcal{B},\mu,T)$.
    \end{enumerate}
\end{proposition}
 \begin{proof}
  If $p_{1}\in\mathbb{Q}[x]$ and $p_{2}\notin\mathbb{Q}[x],$ in which case $p_{2}$ is $\mathbb{Q}$-independent (the case $p_{2}\in\mathbb{Q}[x]$ and $p_{1}\notin\mathbb{Q}[x]$ is analogous), or $p_{1},p_{2}\in\mathbb{Q}[x],$ or $p_{1},p_{2}\notin\mathbb{Q}[x]$ and $p_1,p_2$ are $\mathbb{Q}$-independent, then, by Proposition~\ref{thm2},
  there exists $W_{0}\in\mathbb{N}$ depending only on $p_{1},p_{2}$ such that
  $([p_{1}(n)])_n, ([p_{2}(n)])_n$ are not $W!$-jointly ergodic for $(X,\mathcal{B},\mu,T)$ for all $W\geq W_{0}$.
  
  The remaining case is when $p_{1},p_{2}\notin\mathbb{Q}[x]$ and  $p_{1},p_{2}$ are $\mathbb{Q}$-dependent. Then, we have
  \begin{equation}\label{equ:assumption}
  \begin{split}
   c_{1}p_{1}+c_{2}p_{2}=q,\;\; d_{1}p_{1}+d_{2}p_{2}=q'
  \end{split}
  \end{equation}
for some $q, q'\in\mathbb{Q}[x]$, $c_{1},c_{2}\in S(T)\backslash\mathbb{Z}_\ast$ not both $0$, and $d_{1},d_{2}\in\mathbb{Z}_\ast$  with $\text{gcd}(d_1,d_2)=1$ (since $p_1, p_2\notin \mathbb{Q}[x]$). 
We remark that $d_{1},d_{2}$ and $q'$ are independent of $S(T)$.
For convenience we let
$r:=\det\begin{bmatrix} 
c_{1} & c_{2} \\
d_{1} & d_{2} 
\end{bmatrix}$.
 
There exists $W_{0}\in\mathbb{N}$ depending on $p_{1},p_{2},$ such that for all $W\geq W_{0}$ and $n\in\mathbb{N}$,
$q(W!n)$ is an integer, and $q'(W!n)$ is an integer which is divisible by $d_{1}d_{2}.$\footnote{ This is another reason why we assume that $p_1,p_2$ take the value $0$ at $0.$ If some of $q, q'$ was constant, then the last claim would not be true in general.}

To study the $W!$-joint ergodicity property, we have to consider the following average:
	\begin{equation}\label{targetw}
	\begin{split}
	\lim_{N\to\infty}\frac{1}{N}\sum_{n=1}^{N}e\Bigl(\sum_{i=1}^{2}c_{i}[p_{i}(W!n)]\Bigr),
	\end{split}
	\end{equation}
	 which equals to
	\begin{equation}\nonumber
	\begin{split}
	&
	\lim_{N\to\infty}\frac{1}{N}\sum_{n=1}^{N}e\Bigl(\sum_{i=1}^{2}c_{i}p_{i}(W!n)-\sum_{i=1}^{2}c_{i}\{p_{i}(W!n)\}\Bigr)
	\\&=\lim_{N\to\infty}\frac{1}{N}\sum_{n=1}^{N}e\Bigl(-c_{1}\{p_{1}(W!n)\}-c_{2}\{p_{2}(W!n)\}\Bigr).
	\end{split}
	\end{equation}

	  Since $p_{1}(W!\cdot)\notin\mathbb{Q}[x]$, and $d_{1}p_{1}(W!n)+d_{2}p_{2}(W!n)=q'(W!n)\in d_{1}d_{2}\mathbb{Z}$, we have that the last line of the previous relation is equal to
	\begin{equation}\label{E:18}
	    \lim_{N\to\infty}\frac{1}{N}\sum_{n=1}^{N}e\Bigl(-c_{1}\{p_{1}(W!n)\}-c_{2}\Big\{-\frac{d_{1}}{d_{2}}p_{1}(W!n)\Big\}\Bigr).
	\end{equation}
	For every $n,$ 
	we have that  
	$\Big(\{p_{1}(W!n)\},\Big\{-\frac{d_{1}}{d_{2}}p_{1}(W!n)\Big\}\Big)$ takes value in the subtorus \[Y:=\big\{(\{x\},\{y\}):\;(x,y) \in \mathbb{R}^2, \;d_1x+d_2y=0\big\}.\]
Note that $p_{1}(0)=p_{2}(0)=0$.
 If $\Big(\{p_{1}(W!n)\},\Big\{-\frac{d_{1}}{d_{2}}p_{1}(W!n)\Big\}\Big)_{n}$ is not equidistributed on $Y$ then, by Proposition \ref{green}, there exist $(k_{1},k_{2})\in\mathbb{Z}^{2}\backslash \text{span}_{\mathbb{Q}}\{(d_{1},d_{2})\}$
such that 
\[k_{1}p_{1}(W!n)+k_{2}\Big(-\frac{d_{1}}{d_{2}}p_{1}(W!n)\Big)\in\mathbb{Z}.\]
Since $p_{1}(W!n)\notin\mathbb{Z}$ for some $n$, we have that  $k_{1}d_{2}=k_{2}d_{1},$ a contradiction to the assumption that $(k_{1},k_{2})\notin \text{span}_{\mathbb{Q}}\{(d_{1},d_{2})\}$.
So $\Big(\{p_{1}(W!n)\},\Big\{-\frac{d_{1}}{d_{2}}p_{1}(W!n)\Big\}\Big)_{n}$ is equidistributed on $Y$.
	
Let $F\colon\mathbb{T}^{2}\to\mathbb{C}$ be the function $F(x,y)=e\big(-c_{1}x-c_{2}y\big).$ Then $F$ is Riemann integrable and  \eqref{E:18} is equal to
\begin{equation}\label{E:19}
\int_{Y}e\big(-c_{1}x-c_{2}y\big)\,d m_{Y}(x,y).
\end{equation}
	
Since at least one of the $c_{1},c_{2}$ is irrational, we have that either both are irrational, or one of them is irrational and the other is 0. In both cases, 
By Proposition~\ref{lem:001}, 
\eqref{E:19}, and hence \eqref{targetw} too, is nonzero if  $r\notin\mathbb{Q}_{\ast}$, in which case $([p_{1}(n)])_n, ([p_{2}(n)])_n$ are not $W!$-jointly ergodic for $(X,\mathcal{B},\mu,T)$.

\medskip
 
\noindent The remaining case is when (\ref{equ:assumption}) holds and $r\in\mathbb{Q}_{\ast}$. 
Note that $c_{1},c_{2}\notin\mathbb{Q}$.

Let $g:=\frac{q'}{d_{1}r}$ and $f:=-\frac{q}{r}$.
By (\ref{equ:assumption}),
\[p_{1}=\frac{\det\begin{bmatrix} 
q & c_{2} \\
q' & d_{2} 
\end{bmatrix}}{r}=\frac{\det\begin{bmatrix} 
q & (c_{1}d_{2}-r)/d_{1} \\
q' & d_{2} 
\end{bmatrix}}{r}=-d_{2}\left(f+\Big(c_{1}-\frac{r}{d_{2}}\Big)g\right),\]
and
\[p_{2}=\frac{\det\begin{bmatrix} 
c_{1} & q\\
d_{1} & q' 
\end{bmatrix}}{r}=d_{1}(f+c_{1}g).\]
To complete the proof that $(p_{1},p_{2})$ is of Type-B for $(X,\mathcal{B},\mu,T),$ it suffices to check that $q',$ and thus $g,$ is not constant (hence $0$, since $p_1(0)=p_2(0)=0$). Suppose on the contrary that $q'\equiv 0$. Then, by (\ref{equ:assumption}), we have $p_1=-\frac{d_2}{d_1}p_2$. Again by (\ref{equ:assumption}),
$q=c_1p_1+c_2p_2=-\frac{r}{d_{1}}p_{2},$
which is impossible since $q\in\mathbb{Q}[x]$ and $p_{2}\notin\mathbb{Q}[x]$.
 \end{proof}

\subsection{Proof of Theorem~\ref{cor6}} 

In this subsection, we prove Theorem~\ref{cor6}. Rescaling relevant coefficients and polynomials if necessary, it is not hard to see that   condition (ii) in Theorem~\ref{cor6} is equivalent to the following:

\begin{enumerate}
    \item[$(\text{ii})'$] $p_{1}=u_{1}(f+cg), p_{2}=u_{2}(f+(c+d)g)$ for some $f,g\in\mathbb{Q}[x],$ where $f$ is not a multiple of $g$,  $f(0)=g(0)=0,$ $g\not\equiv 0,$ $c\in \mathbb{R}\backslash\mathbb{Q},$ $u_{1},u_{2}\in\mathbb{Z}$ with $u_{1},u_{2}=\pm 1,$ and
 $d\in\mathbb{Q}_\ast$ such that $dg$ is an integer polynomial.
\end{enumerate}

 In the rest of this section, we will work with condition (ii)$'$ instead of (ii) as it is notation-wise closer to  the one of Type-B pairs.

Since Theorem~\ref{cor6} deals with total joint ergodicity, analogously to \eqref{target} and \eqref{targetw} that we introduced in Subsections~\ref{Ss:main_comp} and ~\ref{Ss:TB} respectively, for $W\in \mathbb{N}$ and $r\in \mathbb{Z},$ we will deal with
\begin{equation}\label{targetwr}
    \lim_{N\to\infty}\frac{1}{N}\sum_{n=1}^{N}e\Bigl(\sum_{i=1}^{2}t_{i}[p_{i}(Wn+r)]\Bigr).
\end{equation}

\subsubsection{Sufficiency}
We first prove the sufficiency part of Theorem~\ref{cor6}.
Assume that both of its conditions fail. 
We start by reducing the problem to the Type-B case.
By the failure of condition (i), $p_{1},p_{2}$ are $\mathbb{R}\backslash\mathbb{Q}_{\ast}$-dependent,
so $c_{1}p_{1}+c_{2}p_{2}\in \mathbb{Q}[x]$ for some $c_{1},c_{2}\in \mathbb{R}\backslash\mathbb{Q}_{\ast},$ not both equal to 0.

By Lemma \ref{construct}, there exists $D\in\mathbb{N}$ and a totally ergodic system $(X,\mathcal{B},\mu,T)$ such that $Dc_{1},Dc_{2}\in S(T)$. Since $Dc_{1},Dc_{2}\notin \mathbb{Q}_{\ast}$,
we have that
$p_{1},p_{2}$ are $S(T)\backslash\mathbb{Q}_{\ast}$-dependent. By Proposition~\ref{case2}, either $([p_{1}(n)])_n,$ $([p_{2}(n)])_n$ are not $W!$-jointly ergodic for $(X,\mathcal{B},\mu,T)$ for some sufficiently large $W$ depending on $p_{1},p_{2},$ and thus $([p_{1}(n)])_n,$ $([p_{2}(n)])_n$ are not totally jointly ergodic for $(X,\mathcal{B},\mu,T)$, or 
 there exist $f,g\in\mathbb{Q}[x],$ $f(0)=g(0)=0,$ $g\not\equiv 0,$   $c\in S(T)\backslash\mathbb{Z},$ $u_{1},u_{2}\in\mathbb{Z}_{\ast}$ with $\text{gcd}(u_{1},u_{2})=1,$ and 
 $d\in\mathbb{Q}_\ast$ 
 such that  
\begin{equation*}
\begin{split}
p_{1}=u_{1}(f+cg),\;\; p_{2}=u_{2}(f+(c+d)g).
\end{split}
\end{equation*}
Note that in particular, $c\in\mathbb{R}\backslash\mathbb{Q}$.

\medskip

We will construct another totally ergodic system $(X',\mathcal{B}',\mu',T')$ for which $([p_{1}(n)])_n,$ $([p_{2}(n)])_n$ are not totally jointly ergodic. We have the following cases.

\medskip

\noindent \textbf{Case 1. $f$ is a multiple of $g$.}
Here, we can rewrite the polynomials as 
\[p_1=u_1c'g',\;\;\;p_2=u_2(c'+d')g'\]
for some $g'\in\mathbb{Q}[x],$ $g'(0)=0,$ 
$g'\not\equiv 0,$ $c'\in\mathbb{R}\backslash\mathbb{Q}$, $d'\in\mathbb{Q}_{\ast}$.
Let $t_{1}, t_{2}$ be irrational numbers such that the following conditions are satisfied:
\begin{itemize}
    \item $t_{1}u_{1}c'+t_{2}u_{2}(c'+d')=c'=sc'+t$ (for $s=1$ and $t=0$);
    \item $\frac{t_{2}d'}{u_1c'}\notin \mathbb{Z}$; 
    \item $t_{1}, t_{2}, c'$ and 1 are $\mathbb{Q}$-independent.\footnote{ This can be easily seen as follows: We want to pick $t_2\notin\mathbb{Q}$ such that $t_{1}u_{1}c'+t_{2}u_{2}(c'+d')=c',$ $d't_2\neq kc'u_1$ for all $k\in \mathbb{Z},$ and 
$s_1t_1+s_2t_2+s_3c'\neq s_0$ for all $s_0,s_1,s_2,s_3\in\mathbb{Q}$ not all 0. This is equivalent in picking $t_2$ irrational with 
$(s_2u_1c'-s_1u_2(c'+d'))t_2\neq s_0u_1c'-s_1c'-s_3u_1(c')^2\;\;\text{for all}\;(s_0,s_1,s_2,s_3)\in\mathbb{Q}^4\backslash\{(0,0,0,0)\},\;\text{and} \;d't_2\neq kc'u_1\;\text{for all}\;k\in \mathbb{Z}.$ It is obvious that there are uncountably many such choices.}
\end{itemize}

By Lemma~\ref{construct}, there exists a totally ergodic system
   $(X',\mathcal{B}',\mu',T')$  such that $S(T')=\text{span}_{\mathbb{Z}}\{t_{1},t_{2},c',1\}.$ 
Then, $(p_{1},p_{2})$ is of Type-B for $(X',\mathcal{B}',\mu',T')$. Since $s=1\in \mathbb{Z}$ and $\frac{t-t_2u_2d'}{u_1u_2c'}=-\frac{t_2d'}{u_1c'}\notin\mathbb{Z},$ we have, by Proposition~\ref{P:Case 2} (ii) and Lemma~\ref{abc_general3},  that 
there exists $W_{0}$ depending only on $p_{1},p_{2},t_{1}$ and $t_{2}$ such  
 that $([p_{1}(n)])_n, ([p_{2}(n)])_n$ are not $W!$-jointly ergodic for $(X',\mathcal{B}',\mu',T')$ for all $W\geq W_{0}$. So  $([p_{1}(n)])_n, ([p_{2}(n)])_n$ are  not totally jointly ergodic for $(X',\mathcal{B}',\mu',T')$.

\medskip

\noindent  \textbf{Case 2. $f$ is not a multiple of $g$.} Here we have two further sub-cases. 

\medskip

\noindent \textbf{Case 2.1. $\vert u_{1}u_{2}\vert>1$.}
 
Since $\vert u_{1}u_{2}\vert>1$,
there exists $a\in \mathbb{Q}$ such that $\frac{ad}{u_{1}}\notin\mathbb{Z}_{\ast}$ and $au_{2}d\in\mathbb{Z}$ (e.g., one can simply take $a=\frac{1}{u_{2}d}$).
By Lemma \ref{construct}, there exists a totally ergodic system  $(X',\mathcal{B}',\mu',T')$ such that $S(T')=\text{span}_{\mathbb{Z}}\Big\{c':=\frac{c}{M},1\Big\}$, 
where $M\in\mathbb{N}$ will be chosen later. 
We may then rewrite
$$p_{1}=u_{1}(f+c'g'),\;\; p_{2}=u_{2}(f+(c'+d')g'),$$
where $g'=Mg$ and $d'=\frac{d}{M}$.
Let $a'=aM$, $t_{1}=-\frac{u_{2}}{u_{1}}a'c'=-\frac{aMu_{2}}{u_{1}}c'$ and $t_{2}=a'c'=aMc'$.
  If we pick $M$ so that $\frac{aM}{u_{1}}\in\mathbb{Z}$, then $t_{1},t_{2}\in S(T')$. Note that $t_{1}u_{1}+t_{2}u_{2}=0$.
On the other hand,
note that $\frac{a'd'}{u_{1}}=\frac{ad}{u_{1}}\notin\mathbb{Z}_{\ast}$ and $a'u_{2}d'=au_{2}d\in\mathbb{Z}$. So, by Proposition~\ref{P:Case 1} (iii) and Lemma~\ref{abc_general3}, 
there exists $W_{0}$ depending only on $p_{1},p_{2},t_{1}$ and $t_{2}$ such  
 that $([p_{1}(n)])_n, ([p_{2}(n)])_n$ are not $W!$-jointly ergodic for $(X',\mathcal{B}',\mu',T')$ for all $W\geq W_{0}$. So $([p_{1}(n)])_n, ([p_{2}(n)])_n$ are not totally jointly ergodic for $(X',\mathcal{B}',\mu',T')$.

\medskip

\noindent \textbf{Case 2.2. $\vert u_{1}u_{2}\vert=1$.}
 
In this case, we may assume without loss of generality that $u_{1}=1$ and $u_{2}=\pm 1$. 
 Since condition (ii)$'$ fails, we have that 
  $dg(r)\notin\mathbb{Z}$ for some $r\in\mathbb{Z}$.
 Let $W\in\mathbb{N}$ be such that $f(Wn+r)-f(r),$ $dg(Wn+r)-dg(r)\in\mathbb{Z}$ for all $n\in \mathbb{N}.$ By Lemma \ref{construct}, there exists a totally ergodic system $(X',\mathcal{B}',\mu',T')$ such that $c/d\in S(T')$.

\medskip
 
\noindent We first consider the case $u_{2}=1$.
 Let $t_{2}=-t_{1}=c/d$. Then, 
 (\ref{targetwr}) is equal to  
 \begin{equation}\label{dge2}
     \begin{split}
         &\lim_{N\to\infty}\frac{1}{N}\sum_{\substack{m=Wn+r, \\  1\leq n\leq N}}e\Bigl(-t_{1}\{f(m)+cg(m)\}-t_{2}\{f(m)+(c+d)g(m)\}\\
&\quad\quad\quad\quad\quad\quad\quad\quad\quad\quad\quad\quad\quad\quad\quad\quad\quad +\{(t_{1}+t_{2})(f(m)+cg(m))+t_{2}dg(m)\}\Bigr)
         \\&=\lim_{N\to\infty}\frac{1}{N}\sum_{\substack{m=Wn+r, \\ 1\leq n\leq N}}e\Bigl(\frac{c}{d}(\{f(r)+cg(m)\}-\{f(r)+dg(r)+cg(m)\})+\{cg(m)\}\Bigr).
     \end{split}
 \end{equation}
 By Weyl's criterion, we have that $(\{cg(Wn+r)\})_{n}$ is equidistributed on $\mathbb{T}$.
 
 If $\{f(r)\}+\{dg(r)\}<1$, then (\ref{dge2}) is equal to 
 \begin{equation}\nonumber
     \begin{split}
         & \lim_{N\to\infty}\frac{1}{N}\sum_{\substack{m=Wn+r, \\ 1\leq n\leq N}}e\Bigl(\frac{c}{d}({\bf{1}}_{\{cg(m)\}\in(1-\{f(r)\}-\{dg(r)\},1-\{f(r)\})}-\{dg(r)\})+\{cg(m)\}\Bigr)
         \\&=\int_{0}^{1}e\Bigl(\frac{c}{d}({\bf{1}}_{x\in(1-\{f(r)\}-\{dg(r)\},1-\{f(r)\})}-\{dg(r)\})+x\Bigr)\,dx 
         \\&=\int_{1-\{f(r)\}-\{dg(r)\}}^{1-\{f(r)\}}e\Bigl(\frac{c}{d}(1-\{dg(r)\})+x\Bigr)\,dx+\int_{1-\{f(r)\}}^{2-\{f(r)\}-\{dg(r)\}}e\Bigl(-\frac{c}{d}\{dg(r)\}+x\Bigr)\,dx
         \\&=\frac{e(\frac{c}{d})-1}{2\pi i}e\Big(-\{f(r)\}-\frac{c}{d}\{dg(r)\}\Big)\big(1-e(-\{dg(r)\})\big).
     \end{split}
 \end{equation}
 
  If $\{f(r)\}+\{dg(r)\}\geq 1$, then (\ref{dge2}) is equal to
 \begin{equation}\nonumber
     \begin{split}
         & \lim_{N\to\infty}\frac{1}{N}\sum_{\substack{m=Wn+r, \\ 1\leq n\leq N}}e\Bigl(\frac{c}{d}({\bf{1}}_{\{cg(m)\}\notin(1-\{f(r)\},2-\{f(r)\}-\{dg(r)\})}-\{dg(r)\})+\{cg(m)\}\Bigr)
         \\&=\int_{0}^{1}e\Bigl(\frac{c}{d}({\bf{1}}_{x\notin(1-\{f(r)\},2-\{f(r)\}-\{dg(r)\})}-\{dg(r)\})+x\Bigr)\,dx 
         \\&=\Big(\int_{0}^{1-\{f(r)\}}+\int_{2-\{f(r)\}-\{dg(r)\}}^{1}\Big)e\Bigl(\frac{c}{d}(1-\{dg(r)\})+x\Bigr)\,dx\\
&\quad\quad\quad\quad\quad\quad\quad\quad\quad\quad\quad\quad\quad\quad\quad\quad\quad\quad +\int_{1-\{f(r)\}}^{2-\{f(r)\}-\{dg(r)\}}e\Bigl(-\frac{c}{d}\{dg(r)\}+x\Bigr)\,dx
         \\&=\int_{2-\{f(r)\}-\{dg(r)\}}^{2-\{f(r)\}}e\Bigl(\frac{c}{d}(1-\{dg(r)\})+x\Bigr)\,dx+\int_{1-\{f(r)\}}^{2-\{f(r)\}-\{dg(r)\}}e\Bigl(-\frac{c}{d}\{dg(r)\}+x\Bigr)\,dx
         \\&=\frac{e(\frac{c}{d})-1}{2\pi i}e\Big(-\{f(r)\}-\frac{c}{d}\{dg(r)\}\Big)\big(1-e(-\{dg(r)\})\big).
     \end{split}
 \end{equation}
 In both cases above, since $dg(r)\notin \mathbb{Z}$, the averages are nonzero. So $([p_{1}(n)])_n, ([p_{2}(n)])_n$ are not  totally jointly ergodic for $(X',\mathcal{B}',\mu',T')$.

\medskip
  
\noindent Finally, we consider the case $u_{2}=-1$.
Let $t_{2}=t_{1}=c/d$. Using the fact that $[-x]+[x]=-1$ for all $x\in\mathbb{R}\backslash\mathbb{Z}$, we may invoke the previous computation to conclude that (\ref{targetwr}) is equal to  
 $$\frac{-e(\frac{c}{d})\big(e(\frac{c}{d})-1\big)}{2\pi i}e\Big(-\{f(r)\}-\frac{c}{d}\{dg(r)\}\Big)\big(1-e(-\{dg(r)\})\big),$$
 i.e., it differs from the value of the first case by a the factor   $-e(\frac{c}{d})$. So, we conclude that $([p_{1}(n)])_n,$ $([p_{2}(n)])_n$ are not  totally jointly ergodic for $(X',\mathcal{B}',\mu',T')$.
 
\subsubsection{Necessity} 
To prove the necessity part of Theorem~\ref{cor6}, we deal with the following two cases. 
 
 \medskip
 
\noindent  \textbf{Condition (i) holds.}
 Since $p_{1},p_{2}$ are $\mathbb{R}\backslash\mathbb{Q}_{\ast}$-independent, then, for any totally ergodic system $(X,\mathcal{B},\mu,T)$, $p_{1},p_{2}$ are $S(T)\backslash\mathbb{Z}_{\ast}$-independent and the result follows from Proposition~\ref{thm1}.

\medskip

\noindent  \textbf{Condition (ii)$'$ holds.}
Assume without loss of generality that $u_1=1.$ 
  It suffices to show that for any totally  ergodic system $(X,\mathcal{B},\mu,T)$, $t_{1},t_{2}\in S(T)$ not both in $\mathbb{Z},$ $W\in \mathbb{N},$ and $r\in\mathbb{Z},$ we have that (\ref{targetwr}) equals 0.  
  
  We first assume that $u_{2}=1$.
 Then, (\ref{targetwr}) equals to  
 \begin{equation}\label{dge}
     \begin{split}
         &\lim_{N\to\infty}\frac{1}{N}\sum_{\substack{m=Wn+r, \\ 1\leq n\leq N}}e\Bigl(-t_{1}\{f(m)+cg(m)\}-t_{2}\{f(m)+(c+d)g(m)\}\\
&\quad\quad\quad\quad\quad\quad\quad\quad\quad\quad\quad\quad\quad\quad\quad\quad\quad\quad\quad\quad\quad +\{(t_{1}+t_{2})(f(m)+cg(m))+t_{2}dg(m)\}\Bigr)
         \\&=\lim_{N\to\infty}\frac{1}{N}\sum_{\substack{m=Wn+r, \\ 1\leq n\leq N}}e\Bigl(-(t_{1}+t_{2})\{f(m)+cg(m)\}+\{(t_{1}+t_{2})(f(m)+cg(m))+t_{2}dg(m)\}\Bigr).
     \end{split}
 \end{equation}
 
 We first consider the case $t_{1}+t_{2}\in\mathbb{Q}$. Since $X$ is totally ergodic, we have that $t_{1}+t_{2}\in\mathbb{Z}$.
 Since (\ref{targetwr}) remains unchanged if we replace $t_{1}$ by $t_{1}+k$ for any $k\in\mathbb{Z}$,
 we may then assume without loss of generality that $t_{1}+t_{2}=0.$   Then, (\ref{dge}) equals to
 \[\lim_{N\to\infty}\frac{1}{N}\sum_{\substack{m=Wn+r, \\ 1\leq n\leq N}}e\big(t_{2}dg(m)\big),\]
 which by Weyl's criterion converges to 0 if $t_{2}\notin\mathbb{Q}$.  Actually, the latter holds since if  $t_{2}\in\mathbb{Q}$, then $t_{2}\in\mathbb{Z}$ by total ergodicity. So, we also have $t_{1}\in\mathbb{Z}$, a contradiction.
 
 We next consider the case $t_{1}+t_{2}\notin\mathbb{Q}$.
 Suppose that 
 \[(a+b(t_1+t_2))f+((a+b(t_1+t_2))c+bt_2d)g=a(f+cg)+b((t_{1}+t_{2})(f+cg)+t_{2}dg)\in\mathbb{Q}[x]\]
 for some $a,b\in\mathbb{Q}$. By Lemma \ref{c1}, which can be used as $f$ is not a multiple of $g,$ we have that $a+b(t_{1}+t_{2})\in\mathbb{Q}$. This implies that $b=0,$ which in turn forces $a=0$.   
 Therefore, $\big(\{(f+cg)(Wn+r)\},\{((t_{1}+t_{2})(f+cg)+t_{2}dg)(Wn+r)\}\big)_{n}$ is equidistributed on $\mathbb{T}^{2}$ by  Weyl's criterion. So, (\ref{dge}) equals to
 \[\int_{\mathbb{T}^{2}}e\big(-(t_{1}+t_{2})x+y\big)\,dxdy=0,\] as was to be shown.
 
 In the case where $u_{2}=-1,$ notice that for all $t_{1},t_{2}\in\mathbb{R}$ and $n\in\mathbb{N}$, we have that 
 $$e(t_{1}[p_{1}(n)]+t_{2}[p_{2}(n)])=e(t_{1}[p_{1}(n)]-t_{2}[-p_{2}(n)])\cdot e(-t_{2}).$$
We may use the previous case to conclude the proof, i.e., for any totally  ergodic system $(X,\mathcal{B},\mu,T)$, $t_{1},t_{2}\in S(T)$ not both in $\mathbb{Z},$ and $W\in \mathbb{N}, r\in \mathbb{Z}$, we have that (\ref{targetwr}) equals 0.

\section{Appendix: Two interesting examples}\label{app}

 In this appendix, we provide two examples illustrating how minor changes in the polynomial iterates can essentially affect their joint ergodicity properties. %In particular, we prove Propositions \ref{apd} and  \ref{apd_2}.

\medskip

The first one shows that, for some $a\in \mathbb{R},$ $([p_{1}(n)])_n,$ $([p_{2}(n)])_n$ can behave differently than  $([p_{1}(n)])_n,$ $([p_{2}(n)+a])_n.$ 

\begin{proposition}\label{apd}
There exist $p_{1},p_{2}\in\mathbb{R}[x]$ and $a\in\mathbb{R}$ such that $([p_{1}(n)])_n,$ $([p_{2}(n)])_n$ are totally jointly ergodic for all totally ergodic systems, and there exists 
a totally ergodic system $(X,\mathcal{B},\mu,T)$, such that for all $W\in \mathbb{N}$ and $r\in \mathbb{Z},$  $([p_{1}(Wn+r)])_n,$ $([p_{2}(Wn+r)+a])_n$ are not jointly ergodic for $(X,\mathcal{B},\mu,T)$.
\end{proposition}

\begin{proof}%[Proof of Proposition \ref{apd}]
Let $W\in\mathbb{N},$ $r\in\mathbb{Z},$ $c\in\mathbb{R}\backslash\mathbb{Q},$ and 
 \[p_{1}(n)=n^{2}+cn,\;\;p_{2}(n)=n^{2}+(c+1)n.\]
	By Theorem~\ref{cor6}, $([p_{1}(n)])_n, ([p_{2}(n)])_{n}$ are totally jointly ergodic for all totally ergodic systems. 
	
	On the other hand, let $(X,\mathcal{B},\mu,T)=(\mathbb{T},\mathcal{B},m,T)$, where $m$ is the Haar measure on $\mathbb{T}$ and $Tx=x+c\mod 1$ for all $x\in\mathbb{T}$. Then
	\begin{equation}\label{exex1}
	\begin{split}
	&\lim_{N\to\infty}\frac{1}{N}\sum_{n=1}^{N}e\Bigl(-c[p_{1}(Wn+r)]+c\Big[p_{2}(Wn+r)+\frac{1}{4}\Big]\Bigr)
	\\&=\lim_{N\to\infty}\frac{1}{N}\sum_{n=1}^{N}e\Bigl(\frac{c}{4}+(c+1)\{c(Wn+r)\}-c\Big\{c(Wn+r)+\frac{1}{4}\Big\}\Bigr)
 \\&=\lim_{N\to\infty}\frac{1}{N}\sum_{n=1}^{N}\Bigl({\bf{1}}_{\{c(Wn+r)\}\in\big[0,\frac{3}{4}\big)}e\big(\{c(Wn+r)\}\big)
 +{\bf{1}}_{\{c(Wn+r)\}\in\big[\frac{3}{4},1\big)}e\big(c+\{c(Wn+r)\}\big)\Bigr).
	\end{split}
	\end{equation}
		Since $\big(c(Wn+r)\big)_{n}$ is equidistributed on $\mathbb{T}$ by Weyl's criterion, we have that (\ref{exex1}) is equal to 
    \begin{equation}\nonumber
	\begin{split} 
         \int_{0}^{\frac{3}{4}}e(t)\,dt+\int_{\frac{3}{4}}^{1}e(c+t)\,dt=\frac{1+i}{2\pi i}(e(c)-1)\neq 0.
	\end{split}
	\end{equation}
	So $\big([p_{1}(Wn+r)]\big)_n, \big([p_{2}(Wn+r)+\frac{1}{4}]\big)_{n}$ are not  jointly ergodic for $(X,\mathcal{B},\mu,T)$. 
	\end{proof}

The second example provides two polynomial sequences which, for all $W\geq 2,$ are $W!$-jointly ergodic for every totally ergodic system, but are not jointly ergodic for some totally ergodic system.

\begin{proposition}\label{apd_2}
		There exist  $p_{1},p_{2}\in\mathbb{R}[x]$ such that $([p_{1}(2n)])_n,$ $([p_{2}(2n)])_n$ are totally jointly ergodic for all totally ergodic systems, and there exists 
	a totally ergodic system $(X,\mathcal{B},\mu,T)$, such that $([p_{1}(n)])_n,$ $([p_{2}(n)])_n$ are not jointly ergodic for $(X,\mathcal{B},\mu,T).$ 
\end{proposition} 

\begin{proof}%[Proof of Proposition~\ref{apd_2}]	
Let  $c\in\mathbb{R}\backslash\mathbb{Q},$ and
	\[p_{1}(n)=n^{3}+\frac{cn^{2}}{4},\;\; p_{2}(n)=n^{3}+\frac{(c+1)n^{2}}{4}.\] 
	By Theorem~\ref{cor6},   $([p_{1}(2n)])_{n}$, $([p_{2}(2n)])_{n}$ are totally jointly ergodic for all totally ergodic systems. 
	
	Let $(X,\mathcal{B},\mu,T)=(\mathbb{T},\mathcal{B},m,T)$, where $m$ is the Haar measure on $\mathbb{T}$ and $Tx=x+c\mod 1$ for all $x\in\mathbb{T}$.
	We will show that $([p_{1}(n)])_n, ([p_{2}(n)])_{n}$ are not jointly ergodic for $(X,\mathcal{B},\mu,T).$
	Since  $([p_{1}(2n)])_n, ([p_{2}(2n)])_{n}$ are  jointly ergodic for $(X,\mathcal{B},\mu,T)$, it suffices to show that 
	$([p_{1}(2n+1)])_n, ([p_{2}(2n+1)])_{n}$ are not jointly ergodic for $(X,\mathcal{B},\mu,T).$
	To this end, notice that 
	\begin{equation}\label{exex2}
	\begin{split}
	& \lim_{N\to\infty}\frac{1}{N}\sum_{n=1}^{N}e\Bigl(-c[p_{1}(2n+1)]+c[p_{2}(2n+1)]\Bigr)
	\\&=\lim_{N\to\infty}\frac{1}{N}\sum_{n=1}^{N}e\Bigl((c+1)\Big\{\frac{c(2n+1)^{2}}{4}\Big\}-c\Big\{\frac{c(2n+1)^{2}}{4}+\frac{1}{4}\Big\}\Bigr)
 \\&=\lim_{N\to\infty}\frac{1}{N}\sum_{n=1}^{N}\left({\bf{1}}_{\Big\{\frac{c(2n+1)^{2}}{4}\Big\}\in\big[0,\frac{3}{4}\big)}e\Big(\Big\{\frac{c(2n+1)^{2}}{4}\Big\}\Big)\right.\\
&\quad\quad\quad\quad\quad\quad\quad\quad\quad\quad\quad\quad \left.+{\bf{1}}_{\Big\{\frac{c(2n+1)^{2}}{4}\Big\}\in\big[\frac{3}{4},1\big)}e\Big(c+\Big\{\frac{c(2n+1)^{2}}{4}\Big\}\Big)\right)e\Big(-\frac{c}{4}\Big).
	\end{split}
	\end{equation}
	Since $c\notin\mathbb{Q}$, $\Big(\frac{c(2n+1)^{2}}{4}\Big)_{n}$ is equidistributed on $\mathbb{T}$ by Weyl's criterion. 
 So, \eqref{exex2} is equal to 
 \begin{equation}\nonumber
 	\begin{split} 
         e\Big(-\frac{c}{4}\Big)\Big(\int_{0}^{\frac{3}{4}}e(t)\,dt+\int_{\frac{3}{4}}^{1}e(c+t)\,dt\Big)=\frac{1+i}{2\pi i}e\Big(-\frac{c}{4}\Big)(e(c)-1)\neq 0.
	\end{split}
	\end{equation}
Thus $([p_{1}(n)])_n, ([p_{2}(n)])_{n}$ are not jointly ergodic for $(X,\mathcal{B},\mu,T)$. 
	\end{proof}	
	
\noindent {\bf{Acknowledgements.}} We would like to thank Nikos Frantzikinakis for many helpful comments during the writing of this article, and Chrysostomos Psaroudakis for his comments on the introduction of an earlier version.

\end{document}